\documentclass[a4paper]{article} 
\RequirePackage{fullpage}

\usepackage[usenames,dvipsnames]{xcolor}
\usepackage[colorlinks=true,linkcolor=black,citecolor=BlueViolet,urlcolor=Maroon]{hyperref}
\usepackage{booktabs,array}
\usepackage{subcaption}
\usepackage{enumerate}
\usepackage{charter}
\usepackage{amsthm, amsmath}
\usepackage[euler-digits,small]{eulervm}  

\usepackage{ifthen}
\usepackage{tikz}
\usetikzlibrary{quotes, shapes}
\usetikzlibrary{backgrounds}
\usetikzlibrary{calc}
\tikzstyle{empty vertex}  = [{circle, draw, fill = white, inner sep=0pt, minimum width=3pt}]
\tikzstyle{filled vertex}=[fill=cyan, circle, thin, draw, inner sep=0pt, minimum width=3pt]

\newtheorem{theorem}{Theorem}[section]
\newtheorem{proposition}[theorem]{Proposition}
\newtheorem{lemma}[theorem]{Lemma}
\newtheorem{corollary}[theorem]{Corollary}

\title{Self-complementary (Pseudo-)Split Graphs}
\author{Yixin Cao\thanks{Department of Computing, Hong Kong Polytechnic University, Hong Kong, China. {\tt \{yixin.cao, haowei.chen\} @polyu.edu.hk}, {\tt shenghua.wang@connect.polyu.hk}. 
    {Supported in part by the Hong Kong Research Grants Council (RGC) under grant 15221420, and the National Natural Science Foundation of China (NSFC) under grants 61972330 and 62372394.}
  }
  \and Haowei Chen\footnotemark[1]
  \and Shenghua Wang\footnotemark[1]
}
\date{}

\begin{document}

\maketitle

\begin{abstract}
  We are concerned with split graphs and pseudo-split graphs whose complements are isomorphic to themselves.  These special subclasses of self-complementary graphs are actually the core of self-complementary graphs.  Indeed, we show that all self-complementary graphs with forcibly self-complementary degree sequences are pseudo-split graphs.  We also give formulas to calculate the number of self-complementary (pseudo-)split graphs of a given order, and show that Trotignon's conjecture holds for all self-complementary split graphs.
  % \keywords{self-complementary graph \and split graph \and pseudo-split graph \and degree sequence.}
\end{abstract}

\section{Introduction}

The \emph{complement} of a graph $G$ is a graph defined on the same vertex set of $G$, where a pair of distinct vertices are adjacent if and only if they are not adjacent in $G$.
In this paper, we study the graph that is isomorphic to its complement, hence called \emph{self-complementary}\index{self-complementary}.
The graph of order one is trivially self-complementary.
There is one self-complementary graph of order four and two self-complementary graphs of order five.
Figure~\ref{fig:share-degree-sequence} lists all self-complementary graphs with eight vertices.
A graph is a \emph{split graph} if its vertex set can be partitioned into a clique and an independent set.
The first three of Figure~\ref{fig:share-degree-sequence} are split graphs, and their rendition in Figures~\ref{fig:split-sc-8-vertices}(a--c) highlight the partition.
% It is more convenient to draw them in the way as Figures~\ref{fig:split-sc-8-vertices}(a--c).

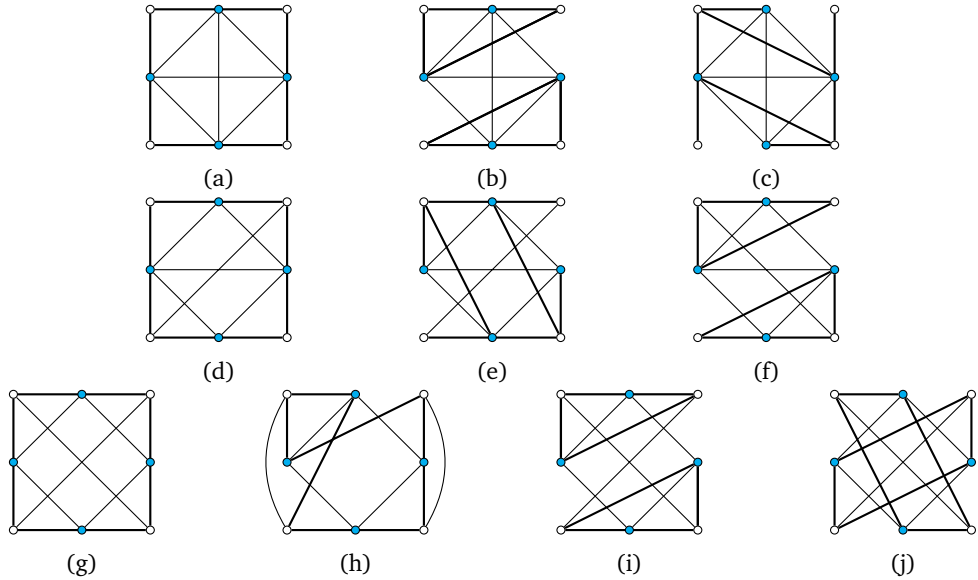
\begin{figure}[ht]
  \centering
  \begin{subfigure}[b]{.22\textwidth}
    \centering
    \begin{tikzpicture}[scale=.6]
      \def\n{4}
      \def\radius{1.5}
      \foreach \i in {0,..., \the\numexpr\n-1\relax} {
        \pgfmathsetmacro{\angle}{45 - (\i - .5) * (360 / \n)}
        \node[filled vertex] (c\i) at (\angle:\radius) {};

        \pgfmathsetmacro\radius{\radius*sqrt(2)}
        \pgfmathsetmacro{\angle}{90 - (\i - .5) * (360 / \n)}
        \node[empty vertex] (i\i) at (\angle:\radius) {};
      }
      \begin{scope}[on background layer]
        \foreach \i in {0,..., \the\numexpr\n-1\relax} {
          \draw let \n1 = {int(Mod(\i - 1, 4))} in (c\n1) -- (c\i);
          \draw[thick] let \n1 = {int(Mod(\i - 1, 4))} in (i\n1) -- (i\i);
        }
        \foreach \i in {1, 2} 
        \draw let \n1 = {int(Mod(\i - 2, 4))} in (c\n1) -- (c\i);
      \end{scope}
    \end{tikzpicture}
    \caption{}
    \label{fig:a}
  \end{subfigure}
  \begin{subfigure}[b]{.22\textwidth}
    \centering
    \begin{tikzpicture}[scale=.6]
      \def\n{4}
      \def\radius{1.5}
      \foreach \i in {0,..., \the\numexpr\n-1\relax} {
        \pgfmathsetmacro{\angle}{45 - (\i - .5) * (360 / \n)}
        \node[filled vertex] (c\i) at (\angle:\radius) {};

        \pgfmathsetmacro\radius{\radius*sqrt(2)}
        \pgfmathsetmacro{\angle}{90 - (\i - .5) * (360 / \n)}
        \node[empty vertex] (i\i) at (\angle:\radius) {};
      }
      \begin{scope}[on background layer]
        \foreach \i in {0,..., \the\numexpr\n-1\relax} {
          \draw let \n1 = {int(Mod(\i - 1, 4))} in (c\n1) -- (c\i);
        }
        \foreach \i in {1, 3} 
        \draw[thick] let \n1 = {int(Mod(\i - 2, 4))} in (i\n1) -- (c\i);
        \foreach \i in {0, 2} {
          \draw[thick] let \n1 = {int(Mod(\i - 1, 4))} in (c\n1) -- (i\i);
        }
        \foreach \i in {1, 3}{
          \pgfmathsetmacro\j{int(Mod(\i + 2, 4))}
          \pgfmathsetmacro\k{int(Mod(\i + 1, 4))}
          \draw[thick] (i\j) -- (c\i) -- (i\k) -- (i\j);
        }
        \foreach \i in {1, 2} 
        \draw let \n1 = {int(Mod(\i - 2, 4))} in (c\n1) -- (c\i);
      \end{scope}
    \end{tikzpicture}
    \caption{}
    \label{fig:b}
  \end{subfigure}
  \begin{subfigure}[b]{.22\textwidth}
    \centering
    \begin{tikzpicture}[scale=.6]
      \def\n{4}
      \def\radius{1.5}
      \foreach \i in {0,..., \the\numexpr\n-1\relax} {
        \pgfmathsetmacro{\angle}{45 - (\i - .5) * (360 / \n)}
        \node[filled vertex] (c\i) at (\angle:\radius) {};

        \pgfmathsetmacro\radius{\radius*sqrt(2)}
        \pgfmathsetmacro{\angle}{90 - (\i - .5) * (360 / \n)}
        \node[empty vertex] (i\i) at (\angle:\radius) {};
      }
      \begin{scope}[on background layer]
        \foreach \i in {0,..., \the\numexpr\n-1\relax} {
          \draw let \n1 = {int(Mod(\i - 1, 4))} in (c\n1) -- (c\i);
        }
        \foreach \i in {0, 2} {
          \draw[thick] let \n1 = {int(\i+1)} in (c\n1) -- (i\i) -- (c\i);
          \draw[thick] let \n1 = {int(Mod(\i - 1, 4))} in (i\n1) -- (i\i);
        }
        \foreach \i in {1, 2} 
        \draw let \n1 = {int(Mod(\i - 2, 4))} in (c\n1) -- (c\i);
      \end{scope}
    \end{tikzpicture}
    \caption{}
    \label{fig:c}
  \end{subfigure}

  \begin{subfigure}[b]{.22\textwidth}
    \centering
    \begin{tikzpicture}[scale=.6]
      \def\n{4}
      \def\radius{1.5}
      \foreach \i in {0,..., \the\numexpr\n-1\relax} {
        \pgfmathsetmacro{\angle}{45 - (\i - .5) * (360 / \n)}
        \node[filled vertex] (c\i) at (\angle:\radius) {};

        \pgfmathsetmacro\radius{\radius*sqrt(2)}
        \pgfmathsetmacro{\angle}{90 - (\i - .5) * (360 / \n)}
        \node[empty vertex] (i\i) at (\angle:\radius) {};
      }
      \begin{scope}[on background layer]
        \foreach \i in {0,..., \the\numexpr\n-1\relax} {
          \draw let \n1 = {int(Mod(\i - 1, 4))} in (c\n1) -- (c\i);
          \draw[thick] let \n1 = {int(Mod(\i - 1, 4))} in (i\n1) -- (i\i);
        }
      \end{scope}
      \draw (c1) -- (c3) (i1) -- (i3);
    \end{tikzpicture}
    \caption{}
    \label{fig:d}
  \end{subfigure}
  \begin{subfigure}[b]{.22\textwidth}
    \centering
    \begin{tikzpicture}[scale=.6]
      \def\n{4}
      \def\radius{1.5}
      \foreach \i in {0,..., \the\numexpr\n-1\relax} {
        \pgfmathsetmacro{\angle}{45 - (\i - .5) * (360 / \n)}
        \node[filled vertex] (c\i) at (\angle:\radius) {};

        \pgfmathsetmacro\radius{\radius*sqrt(2)}
        \pgfmathsetmacro{\angle}{90 - (\i - .5) * (360 / \n)}
        \node[empty vertex] (i\i) at (\angle:\radius) {};
      }
      \begin{scope}[on background layer]
        \foreach \i in {0,..., \the\numexpr\n-1\relax} {
          \draw let \n1 = {int(Mod(\i - 1, 4))} in (c\n1) -- (c\i);
        }
        \foreach \i in {0, 2} {
          \draw[thick] let \n1 = {int(2 -\i)} in (c\n1) -- (i\i);
          \draw[thick] let \n1 = {int(Mod(\i - 1, 4))} in (c\n1) -- (i\i);
        }
        \draw[thick] (i0) -- (i1) (i3) -- (i2);
        \draw (c1) -- (c3)(i1) -- (i3);;
      \end{scope}
    \end{tikzpicture}
    \caption{}
    \label{fig:e}
  \end{subfigure}
  \begin{subfigure}[b]{.22\textwidth}
    \centering
    \begin{tikzpicture}[scale=.6]
      \def\n{4}
      \def\radius{1.5}
      \foreach \i in {0,..., \the\numexpr\n-1\relax} {
        \pgfmathsetmacro{\angle}{45 - (\i - .5) * (360 / \n)}
        \node[filled vertex] (c\i) at (\angle:\radius) {};

        \pgfmathsetmacro\radius{\radius*sqrt(2)}
        \pgfmathsetmacro{\angle}{90 - (\i - .5) * (360 / \n)}
        \node[empty vertex] (i\i) at (\angle:\radius) {};
      }
      \begin{scope}[on background layer]
        \foreach \i in {0,..., \the\numexpr\n-1\relax} {
          \draw let \n1 = {int(Mod(\i - 1, 4))} in (c\n1) -- (c\i);
        }
        \foreach \i in {1, 3} 
        \draw[thick] let \n1 = {int(Mod(\i - 2, 4))} in (i\n1) -- (c\i);
        \foreach \i in {0, 2} {
          \draw[thick] let \n1 = {int(Mod(\i - 1, 4))} in (c\n1) -- (i\i);
        }
        \draw (c1) -- (c3) (i0) -- (i2);
        \draw[thick] (i1) -- (i0) (i2) -- (i3);
      \end{scope}
    \end{tikzpicture}
    \caption{}
    \label{fig:f}
  \end{subfigure}

  \begin{subfigure}[b]{.22\textwidth}
    \centering
    \begin{tikzpicture}[scale=.6]
      \def\n{4}
      \def\radius{1.5}
      \foreach \i in {0,..., \the\numexpr\n-1\relax} {
        \pgfmathsetmacro{\angle}{45 - (\i - .5) * (360 / \n)}
        \node[filled vertex] (c\i) at (\angle:\radius) {};

        \pgfmathsetmacro\radius{\radius*sqrt(2)}
        \pgfmathsetmacro{\angle}{90 - (\i - .5) * (360 / \n)}
        \node[empty vertex] (i\i) at (\angle:\radius) {};
      }
      \begin{scope}[on background layer]
        \foreach \i in {0,..., \the\numexpr\n-1\relax} {
          \draw let \n1 = {int(Mod(\i - 1, 4))} in (c\n1) -- (c\i);
          \draw[thick] let \n1 = {int(Mod(\i - 1, 4))} in (i\n1) -- (i\i);
        }
        \foreach \i in {1, 2} 
        \draw let \n1 = {int(Mod(\i - 2, 4))} in (i\n1) -- (i\i);
      \end{scope}
    \end{tikzpicture}
    \caption{}
    \label{fig:g}
  \end{subfigure}
  \begin{subfigure}[b]{.22\textwidth}
    \centering
    \begin{tikzpicture}[scale=.6]
      \def\n{4}
      \def\radius{1.5}
      \foreach \i in {0,..., \the\numexpr\n-1\relax} {
        \pgfmathsetmacro{\angle}{45 - (\i - .5) * (360 / \n)}
        \node[filled vertex] (c\i) at (\angle:\radius) {};

        \pgfmathsetmacro\radius{\radius*sqrt(2)}
        \pgfmathsetmacro{\angle}{90 - (\i - .5) * (360 / \n)}
        \node[empty vertex] (i\i) at (\angle:\radius) {};
      }
      \begin{scope}[on background layer]
        \foreach \i in {0,..., \the\numexpr\n-1\relax} {
          \draw let \n1 = {int(Mod(\i - 1, 4))} in (c\n1) -- (c\i);
        }
        \draw[thick] (i1) -- (i2) -- (i3) (i1) -- (c3) (i3) -- (c0) -- (i0) -- (c3);
        \draw[bend left] (i1) edge (i2) (i3) edge (i0);
      \end{scope}
    \end{tikzpicture}
    \caption{}
    \label{fig:h}
  \end{subfigure}
  \begin{subfigure}[b]{.22\textwidth}
    \centering
    \begin{tikzpicture}[scale=.6]
      \def\n{4}
      \def\radius{1.5}
      \foreach \i in {0,..., \the\numexpr\n-1\relax} {
        \pgfmathsetmacro{\angle}{45 - (\i - .5) * (360 / \n)}
        \node[filled vertex] (c\i) at (\angle:\radius) {};

        \pgfmathsetmacro\radius{\radius*sqrt(2)}
        \pgfmathsetmacro{\angle}{90 - (\i - .5) * (360 / \n)}
        \node[empty vertex] (i\i) at (\angle:\radius) {};
      }
      \begin{scope}[on background layer]
        \foreach \i in {0,..., \the\numexpr\n-1\relax} {
          \draw let \n1 = {int(Mod(\i - 1, 4))} in (c\n1) -- (c\i);
        }
        \foreach \i in {1, 3}{
          \pgfmathsetmacro\j{int(Mod(\i + 2, 4))}
          \pgfmathsetmacro\k{int(Mod(\i + 1, 4))}
          \draw[thick] (i\j) -- (c\i) -- (i\k) -- (i\j);
        }
        \foreach \i in {0, 2} {
          \draw[thick] let \n1 = {int(Mod(\i - 1, 4))} in (c\n1) -- (i\i);
        }
        \draw (i1) -- (i0) -- (i2) -- (i3) -- (i1);
      \end{scope}
    \end{tikzpicture}
    \caption{}
    \label{fig:i}
  \end{subfigure}
  \begin{subfigure}[b]{.22\textwidth}
    \centering
    \begin{tikzpicture}[scale=.6]
      \def\n{4}
      \def\radius{1.5}
      \foreach \i in {0,..., \the\numexpr\n-1\relax} {
        \pgfmathsetmacro{\angle}{45 - (\i - .5) * (360 / \n)}
        \node[filled vertex] (c\i) at (\angle:\radius) {};

        \pgfmathsetmacro\radius{\radius*sqrt(2)}
        \pgfmathsetmacro{\angle}{90 - (\i - .5) * (360 / \n)}
        \node[empty vertex] (i\i) at (\angle:\radius) {};
      }
      \begin{scope}[on background layer]
        \foreach \i in {0,..., \the\numexpr\n-1\relax} {
          \draw let \n1 = {int(Mod(\i - 1, 4))} in (c\n1) -- (c\i);
        }
        \foreach \i in {0,..., 3} {
          \draw[thick] let \n1 = {int(Mod(\i - 2, 4))} in (i\i) -- (c\i) -- (i\n1);
        }
        \foreach \i in {1, 2} 
        \draw let \n1 = {int(Mod(\i - 2, 4))} in (i\n1) -- (i\i);
      \end{scope}
    \end{tikzpicture}
    \caption{}
    \label{fig:j}
  \end{subfigure}
  \caption{All self-complementary graphs on eight vertices.
    In each graph, the four vertices with lower degree are represented as empty nodes, and others filled nodes.}
  \label{fig:share-degree-sequence}
\end{figure}

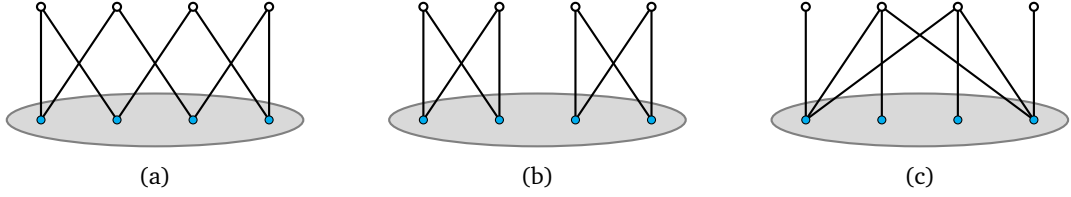
\begin{figure}[ht]
  \centering
  \tikzset{every path/.style={thick}}
  \begin{subfigure}[b]{.3\textwidth}
    \centering
    \begin{tikzpicture}
      \draw[fill = gray!30, draw = gray] (2.5, 0) ellipse (1.95 and 0.35);
      \foreach \i in {2, 3} 
      \draw (\i - 1, 0) -- (\i, 1.5) -- (\i + 1, 0);
      \foreach \i in {1, 4}
      \draw (\i, 0) -- (\i, 1.5) -- ({\i - (-1)^\i}, 0);
      \foreach \i in {1, ..., 4} {
        \node[filled vertex] (c\i) at (\i, 0) {};
        \node[empty vertex] (i\i) at (\i, 1.5) {};
      }
    \end{tikzpicture}
    \caption{}
    \label{fig:split-sc-8-vertices-3}
  \end{subfigure}
  \,
  \begin{subfigure}[b]{.3\textwidth}
    \centering
    \begin{tikzpicture}
      \draw[fill = gray!30, draw = gray] (2.5, 0) ellipse (1.95 and 0.35);
      \foreach \i in {1,..., 4}
      \draw (\i, 0) -- (\i, 1.5) -- ({\i - (-1)^\i}, 0);
      \foreach \i in {1, ..., 4} {
        \node[filled vertex] (c\i) at (\i, 0) {};
        \node[empty vertex] (i\i) at (\i, 1.5) {};
      }
    \end{tikzpicture}
    \caption{}
    \label{fig:split-sc-8-vertices-1}
  \end{subfigure}
  \,
  \begin{subfigure}[b]{.3\textwidth}
    \centering
    \begin{tikzpicture}
      \draw[fill = gray!30, draw = gray] (2.5, 0) ellipse (1.95 and 0.35);
      \foreach \i in {2, 3} 
      \draw (1, 0) -- (\i, 1.5) -- (4, 0);
      \foreach \i in {1, ..., 4} 
      \draw (\i, 0) -- (\i, 1.5);
      \foreach \i in {1, ..., 4} {
        \node[filled vertex] (c\i) at (\i, 0) {};
        \node[empty vertex] (i\i) at (\i, 1.5) {};
      }
    \end{tikzpicture}
    \caption{}
    \label{fig:split-sc-8-vertices-2}
  \end{subfigure}
  \caption{Self-complementary split graphs with eight vertices.  Vertices in $I$ are represented by empty nodes on the top, while vertices in $K$ are represented by filled nodes on the bottom. For clarity, edges among vertices in $K$ are omitted.
    The degree sequences are (a) $(5^4, 2^4)$, (b) $(5^4, 2^4)$, and (c) $(6^2,4^2,3^2,1^2)$.}
  \label{fig:split-sc-8-vertices}
\end{figure}

These two families of graphs are connected by the following observation.
An elementary counting argument convinces us that the order of a nontrivial self-complementary graph is either $4k$ or $4k+1$ for some positive integer $k$.
Consider a self-complementary graph $G$ of order $4k$, where $L$ (resp., $H$) represents the set of $2k$ vertices with smaller (resp., higher) degrees.  Note that $d(x) \le 2 k - 1 < 2k \le d(y)$ for every pair of vertices $x\in L$ and $y\in H$.
Xu and Wong~\cite{xu-00-self-complementary-graphs} observed that the subgraphs of $G$ induced by $L$ and $H$ are complementary to each other.  More importantly, the bipartite graph spanned by the edges between $L$ and $H$ is closed under \emph{bipartite complementation}, i.e., reversing edges in between but keeping both $L$ and $H$ independent. See the thick edges in Figure~\ref{fig:share-degree-sequence}.
When studying the connection between $L$ and $H$, it is more convenient to add all the missing edges among $H$ and remove all the edges among $L$, thereby turning $G$ into a self-complementary split graph.
In this sense, every self-complementary graph of order $4 k$ can be constructed from a self-complementary split graph of the same order and a graph of order $2 k$.
For a self-complementary graph of an odd order, the self-complementary split graph is replaced by a self-complementary pseudo-split graph.
A pseudo-split graph is either a split graph or a split graph plus a five-cycle such that every vertex on the cycle is adjacent to every vertex in the clique of the split graph and is nonadjacent to any vertex in the independent set of the split graph.

The decomposition theorem of Xu and Wong~\cite{xu-00-self-complementary-graphs} was for the construction of self-complementary graphs, another ingredient of which is the degree sequences of these graphs (the non-increasing sequence of its vertex degrees).
Clapham and Kleitman~\cite{clapham-76-self-complementary-degree-sequences, clapham-76-potentially-self-complementary-sequences} present a necessary condition for a degree sequence to be that of a self-complementary graph.
However, a realization of such a degree sequence may or may not be self-complementary.
A natural question is to ask about the degree sequences all of whose realizations are necessarily self-complementary, called \emph{forcibly self-complementary}.
All the degree sequences for self-complementary graphs up to order five, $(0^1)$, $(2^2, 1^2)$, $(2^5)$, and $(3^2, 2^1,1^2)$, are forcibly self-complementary.
Of the four degree sequences for the self-complementary graphs of order eight, only $(5^4,2^4)$ and $(6^2,4^2,3^2,1^2)$ are focibly self-complementary.  All the realizations of these forcibly self-complementary degree sequences turn out to be pseudo-split graphs.  As we will see, this is not incidental.

We take $p$ graphs $S_{1}$, $S_{2}$, $\ldots$, $S_{p}$, each being either a four-path or one of the first two graphs in Figure~\ref{fig:share-degree-sequence}.
Note that the each of them admits a unique decomposition into a clique $K_{i}$ and an independent set $I_{i}$.
For any pair of $i, j$ with $1\le i < j \le p$, we add all possible edges between $K_{i}$ and $K_{j}\cup I_{j}$.
It is easy to verify that the resulting graph is self-complementary, and can be partitioned into a clique $\bigcup_{i=1}^{p} K_{i}$ and an independent set $\bigcup_{i=1}^{p} I_{i}$.
By an \emph{elementary self-complementary pseudo-split graph} we mean such a graph, or one obtained from it by adding a single vertex or a five-cycle and make them complete to $\bigcup_{i=1}^{p} K_{i}$.
For example, we end with the graph in Figure~\ref{fig:share-degree-sequence}(c)
with $p = 2$ and both $S_{1}$ and $S_{2}$ being four-paths.
It is a routine exercise to verify that the degree sequence of an elementary self-complementary pseudo-split graph is forcibly self-complementary.  We show that the other direction holds as well, thereby fully characterizing forcibly self-complementary degree sequences.

\begin{theorem}\label{thm:main}
   A degree sequence is forcibly self-complementary if and only if every realization of it is an elementary self-complementary pseudo-split graph.
\end{theorem}

Our result also bridges a longstanding gap in the literature on self-complementary graphs.
Rao~\cite{rao1981survey} has proposed another characterization for forcibly self-complementary degree sequences (we leave the statement, which is too technical, to Section~\ref{sec:degree-sequences}).  As far as we can check, he never published a proof of his characterization.  It follows immediately from Theorem~\ref{thm:main}.

All self-complementary graphs up to order five are pseudo-split graphs, while only three out of the ten self-complementary graphs of order eight are.
By examining the list of small self-complementary graphs, Ali~\cite{ali-08-thesis} counted self-complementary split graphs up to $17$ vertices.
Whether a graph is a split graph can be determined solely by its degree sequence.
However, this approach needs the list of all self-complementary graphs, and hence cannot be generalized to large graphs.
Answering a question of Harary~\cite{harary:1960}, Read~\cite{read-63-self-complementary-graphs} presented a formula for the number of self-complementary graphs with a specific number of vertices.
Clapham~\cite{clapham-84-enumeration} simplified Read's formula by studying the isomorphisms between a self-complementary graph and its complement.
We take an approach similar to Clapham's for self-complementary split graphs with an even order, which leads to a formula for the number of such graphs.
For other self-complementary pseudo-split graphs,
we establish a one-to-one correspondence between self-complementary split graphs on $4k$ vertices and those on $4k+1$ vertices, and a one-to-one correspondence between self-complementary pseudo-split graphs of order $4k+1$ that are not split graphs and self-complementary split graphs on $4k-4$ vertices.

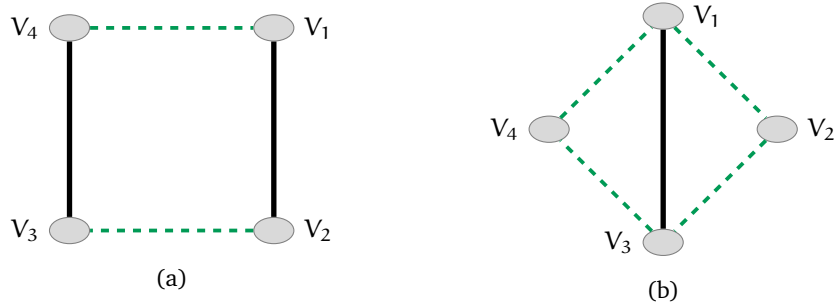
\begin{figure}[ht]
  \tikzstyle{full edge}=[line width=0.70mm]
  \tikzstyle{partial edge}=[line width=0.50mm, ForestGreen, dashed]
  \tikzstyle{model vertex}  = [{ellipse, fill = gray!30, draw = gray, inner sep=3.5pt, minimum width=15pt}]
\centering
  \begin{subfigure}[ht]{.4\textwidth}
    \centering
    \begin{tikzpicture}
      \foreach \i in {3,4} 
        \node[model vertex, "$V_{\i}$" left] (\i) at ({135 - 90*\i}:1.9) {};
      \foreach \i in {1, 2} 
        \node[model vertex, "$V_{\i}$" right] (\i) at ({135 - 90*\i}:1.9) {};
      \draw[full edge] (1) -- (2) (3) -- (4);
      \draw[partial edge] (2) -- (3) (4) -- (1);
    \end{tikzpicture}
    \caption{}
    \label{fig:partition-rectangle-model}
  \end{subfigure}
    \begin{subfigure}[ht]{.4\textwidth}
    \centering
    \begin{tikzpicture}
      \foreach \i in {1, 2} 
        \node[model vertex, "$V_{\i}$" right] (\i) at ({180 - 90*\i}:1.5) {};
      \foreach \i in {3,4} 
        \node[model vertex, "$V_{\i}$" left] (\i) at ({180 - 90*\i}:1.5) {};
      \draw[full edge] (1) -- (3);
      \draw[partial edge] (1) -- (2) (2) -- (3) (3) -- (4) (4) -- (1);
    \end{tikzpicture}
    \caption{}
    \label{fig:partition-diamond-model}
  \end{subfigure}
  \caption{The (a) rectangle and (b) diamond partitions.  Each node represents one part of the partition.
    A solid line indicates that all the edges between the two parts are present, a missing line indicates that there is no edge between the two parts, while a dashed line imposes
    no restrictions on the two parts.}
  \label{fig:model}
\end{figure}

We also study a conjecture of 
Trotignon \cite{trotignon-05-self-complementary-graphs}, which asserts that if a self-complementary graph $G$ does not contain a five-cycle, then its vertex set can be partitioned into four nonempty sets with the adjacency patterns of a rectangle or a diamond, as described in Figure~\ref{fig:model}. 
He managed to prove that certain special graphs satisfy this conjecture.
The study of rectangle partitions in self-complementary graphs enabled Trotignon to present a new proof of Gibbs' theorem~\cite[Theorem 4]{gibbs-74-self-complementary}.
We prove Trotignon's conjecture on self-complementary split graphs, with a stronger statement.
We say that a partition of $V(G)$ is \emph{self-complementary} if it forms the same partition in the complement of $G$, illustrated in Figure~\ref{fig:selfg-contains-C5}.
Every self-complementary split graph of an even order admits a diamond partition that is self-complementary.
Moreover, for each positive integer $k$, there is a single graph of order $4k$ that admits a rectangle partition.
Note that under the setting as Trotignon's, the graph always admits a partition that is self-complementary, while in general, there are graphs that admit a partition, but do not admit any partition that is self-complementary \cite{cao-23-self-complementary-partition}.

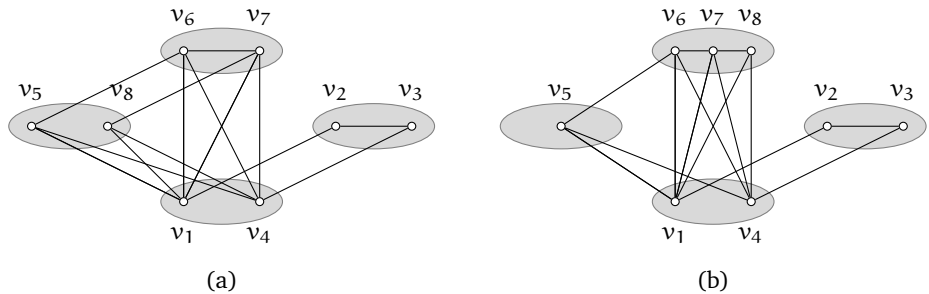
\begin{figure}[ht]
  \centering
  \begin{subfigure}[b]{.4\textwidth}
    \centering
    \begin{tikzpicture}[scale = 1]
      \foreach \i in {1,3} 
      \draw[fill = gray!30, draw = gray] ({180 - 90*\i}:1) ellipse (.8 and 0.3);
      \foreach \i in {2,4} 
      \draw[fill = gray!30, draw = gray] ({180 - 90*\i}:2) ellipse (.8 and 0.3);

      \node at (-0.5,1.45) {$v_{6}$};
      \node at (0.5,1.45) {$v_{7}$};
      \node at (-0.5,-1.45) {$v_{1}$};
      \node at (0.5,-1.45) {$v_{4}$};
      
      \node at (-2.5,0.45) {$v_{5}$};
      \node at (-1.3,0.45) {$v_{8}$};
      \node at (1.5,0.45) {$v_{2}$};
      \node at (2.5,0.45) {$v_{3}$};

      \node[empty vertex] (1) at (-0.5,-1) {};
      \node[empty vertex] (4) at (0.5,-1) {};
      \node[empty vertex] (5) at (-2.5,0) {};
      \node[empty vertex] (8) at (-1.5,0) {};
      \node[empty vertex] (6) at (-0.5,1) {};
      \node[empty vertex] (7) at (0.5,1) {};
      \node[empty vertex] (2) at (1.5,0) {};
      \node[empty vertex] (3) at (2.5,0) {};
      \draw (1) -- (5) (1) -- (6) (1) -- (7);
      \draw[black] (1) -- (2) -- (3) -- (4);
      \draw (5) -- (6) -- (7) -- (8);
      \draw (5) -- (1) (5) -- (4);
      \draw (6) -- (1) (6) -- (4);
      \draw (7) -- (1) (7) -- (4);
      \draw[black] (8) -- (1) (8) -- (4);
    \end{tikzpicture}
    \caption{}
  \end{subfigure}
  \begin{subfigure}[b]{.4\textwidth}
    \centering
    \begin{tikzpicture}[scale = 1]
      \foreach \i in {1,3} 
      \draw[fill = gray!30, draw = gray] ({180 - 90*\i}:1) ellipse (.8 and 0.3);
      \foreach \i in {2,4} 
      \draw[fill = gray!30, draw = gray] ({180 - 90*\i}:2) ellipse (.8 and 0.3);

      \node at (-0.5,1.45) {$v_{6}$};
      \node at (0,1.45) {$v_{7}$};
      \node at (0.5,1.45) {$v_{8}$};
      \node at (-0.5,-1.45) {$v_{1}$};
      \node at (0.5,-1.45) {$v_{4}$};
      
      \node at (-2,0.45) {$v_{5}$};
      
      \node at (1.5,0.45) {$v_{2}$};
      \node at (2.5,0.45) {$v_{3}$};

      \node[empty vertex] (1) at (-0.5,-1) {};
      \node[empty vertex] (4) at (0.5,-1) {};
      \node[empty vertex] (5) at (-2,0) {};
      \node[empty vertex] (8) at (0.5,1) {};
      \node[empty vertex] (6) at (-0.5,1) {};
      \node[empty vertex] (7) at (0,1) {};
      \node[empty vertex] (2) at (1.5,0) {};
      \node[empty vertex] (3) at (2.5,0) {};
      \draw (1) -- (5) (1) -- (6) (1) -- (7);
      \draw[black] (1) -- (2) -- (3) -- (4);
      \draw (5) -- (6) -- (7) -- (8);
      \draw (5) -- (1) (5) -- (4);
      \draw (6) -- (1) (6) -- (4);
      \draw (7) -- (1) (7) -- (4);
      \draw[black] (8) -- (1) (8) -- (4);
    \end{tikzpicture}
    \caption{}
  \end{subfigure}
  \caption{Two diamond partitions of a self-complementary graph; only the first is self-complementary.}
  \label{fig:selfg-contains-C5}
\end{figure}

Before closing this section, let us mention related work.
There is another natural motivation to study self-complementary split graphs.  Sridharan and Balaji~\cite{sridharan-98-self-complementary-chordal} tried to understand self-complementary graphs that are chordal.
  They are precisely split graphs~\cite{foldes-77-split-graphs}.
The class of split graphs is \emph{closed under complementation}.\footnote{Some authors call such graph classes ``self-complementary,'' e.g., the influential ``Information System on Graph Classes and their Inclusions'' (https://www.graphclasses.org).}
We may study self-complementary graphs in other graph classes.  Again, for this purpose, it suffices to focus on those closed under complementation.
% A hereditary graph class can be defined by a (possibly infinite) family of forbidden induced subgraphs.
In the simplest case, we can define such a class by forbidding a graph $F$ as well as its complement.
It is not interesting when $F$ consists of two or three vertices, or when it is the four-path.
When $F$ is the four-cycle, we end with the class of pseudo-split graphs, which is the simplest in this sense.
A more important class closed under complementation is perfect graphs.
We leave it open to characterize self-complementary perfect graphs.
Another open problem is the recognition of self-complementary (pseudo)-split graphs.
It is well known that the isomorphism test of both self-complementary graphs and (pseudo)-split graphs are GI-complete
\cite{colbourn-78-isomorphism-and-self-complementary, lueker-79-interval-isomorphism}.

% After recalling basic definitions and facts in Section~\ref{sec:pre}, we prove Theorem~\ref{thm:main} in Section~\ref{sec:degree-sequences}.  We present the formula for enumerating self-complementary (pseudo-)split graphs in Section~\ref{sec:enumeration} and prove Theorem~\ref{thm:scs-diamond} in Section~\ref{sec:self-complementary partitions}.  The last three sections, relying only on Section~\ref{sec:pre}, can be read independently.

\section{Preliminaries}\label{sec:pre}

All the graphs discussed in this paper are finite and simple. The vertex set and edge set of a graph $G$ are denoted by, respectively, $V(G)$ and $E(G)$.
The two ends of an edge are \emph{neighbors} of each other, and the number of neighbors of a vertex $v$, denoted by $d_{G}(v)$, is its \emph{degree}.  We may drop the subscript $G$ if the graph is clear from the context.
For a subset $U \subseteq V(G)$, let $G[U]$ denote the subgraph of $G$ induced by $U$, whose vertex set is $U$ and whose edge set comprises all the edges with both ends in $U$, and let $G - U = G[V(G) \setminus U]$, which is simplified to $G - u$ if $U$ comprises a single vertex $u$. A \emph{clique} is a set of pairwise adjacent vertices, and an \emph{independent set} is a set of vertices that are pairwise nonadjacent.
For $\ell \geq 1$, we use $P_\ell$ and $K_\ell$ to denote the path graph and the complete graph, respectively, on $\ell$ vertices. For $\ell \geq 3$, we use $C_\ell$ to denote the $\ell$-cycle.
We say that two sets of vertices are \emph{complete} or \emph{nonadjacent} to each other if there are all possible edges or if there is no edge between them, respectively.

A graph is a \emph{split graph} if its vertex set can be partitioned into a clique and an independent set.  We use $K \uplus I$, where $K$ being a clique and $I$ an independent set, to denote a \emph{split partition} of a split graph.
The following is straightforward.  Since it is not used in the present paper, we omit the proof.
\begin{proposition}\label{lem:split-core}
  Let $G$ be a graph of an order $4 k$, and let $H$ and $L$ be the $2k$ vertices of the largest and smallest degrees, respectively.
  If $G$ is self-complementary, then it remains a self-complementary after $H$ replaced by a clique and $L$ an independent set.
\end{proposition}

An \emph{isomorphism}\index{isomorphism} between two graphs $G_{1}$ and $G_{2}$ is a bijection between their vertex sets, i.e., $\sigma\colon V(G_{1})\to V(G_{2})$, such that two vertices $u$ and $v$ are adjacent in $G_{1}$ if and only if $\sigma(u)$ and $\sigma(v)$ are adjacent in $G_{2}$. 
Two graphs with an isomorphism are \emph{isomorphic}.
A graph is \emph{self-complementary} if it is isomorphic to its \emph{complement} $\overline{G}$, the graph defined on the same vertex set of $G$, where a pair of distinct vertices are adjacent in $\overline{G}$ if and only if they are not adjacent in $G$.
An isomorphism between $G$ and $\overline{G}$ is a permutation of $V(G)$, called an \emph{antimorphism}.

A split graph may have more than one split partition; e.g., a complete graph on ${n}$ vertices has $n + 1$ different split partitions.  
\begin{lemma}\label{lem:unique-split-partition}
  A self-complementary split graph on $4 k$ vertices has a unique split partition and it is
    \begin{equation}
    \label{eq:1}
    \left\{ v\mid d(v) \ge 2k \right\}
    \uplus \left\{ v\mid d(v) < 2k \right\}.
  \end{equation}
\end{lemma}  
\begin{proof}
  Let $G$ be a self-complementary split graph with $4k$ vertices, and $\sigma$ an antimorphism of $G$.
  By definition, for any vertex $v\in V(G)$, we have $d(v) + d(\sigma(v)) = 4 k - 1$.  Thus,
  \[
    \min (d(v), d(\sigma(v))) \le 2 k - 1 < 2k \le \max (d(v), d(\sigma(v))).
  \]
  As a result, $G$ does not contain any clique or independent set of order $2k + 1$.
  Suppose for contradictions that there exists a split partition $K\uplus I$ of $G$ different from \eqref{eq:1}.
  There must be a vertex $x\in I$ with $d(x) \ge 2k$.
  We must have $d(x) = 2 k$ and $N(x) \subseteq K$.
  But then there are at least $|N[x]| = 2 k + 1$ vertices having degree at least $2 k$, a contradiction.
\end{proof}

We represent an antimorphism as the product of disjoint cycles 
$\sigma = \sigma_{1} \sigma_{2} \cdots \sigma_{p}$,
where $\sigma_{i} = (v_{i 1} v_{i 2} \cdots)$ for all $i$.
Sachs and Ringel~\cite{sachs-62-self-complementary-graphs,ringel-63-self-complementary} independently showed that there can be at most one vertex $v$ fixed by an antimorphism $\sigma$, i.e., $\sigma(v) = v$.
For any other vertex $u$, the smallest number $k$ satisfying $\sigma^k(u) = u$ has to be a multiplier of four. Gibbs~\cite{gibbs-74-self-complementary} observed that if a vertex $v$ has $d$ neighbors in $G$, then the degree of $\sigma(v)$ in $G$ is $n - 1 - d$ where $n$ is the order of $G$. It implies that if $v$ is fixed by $\sigma$, then its degree in $G$ is $(n-1)/2$. Furthermore, the vertices in every cycle of $\sigma$ with a length of more than one alternate in degrees $d$ and $n-1-d$.

\begin{lemma}[{\cite{sachs-62-self-complementary-graphs,ringel-63-self-complementary}}]
  \label{lem:length-of-circular-permutation}
  If $\sigma$ is an antimorphism of a self-complementary graph, then the length of each cycle in $\sigma$ is either $1$ or a multiplier of $4$.
\end{lemma}

For any subset of cycles in $\sigma$, the vertices within those cycles induce a subgraph that is self-complementary.  Indeed, the selected cycles themselves act as an antimorphism for the subgraph.

\begin{proposition}[\cite{gibbs-74-self-complementary}]
  \label{prop:every-subset-of-cycles-are-scs}
  Let $G$ be a self-complementary graph and $\sigma$ an antimorphism of $G$. For any subset of cycles in $\sigma$, the vertices within those cycles induce a self-complementary graph.
\end{proposition}

The following observation correlate self-complementary split graphs having even and odd orders.

\begin{proposition}\label{prop:scs-odd}
  Let $G$ be a split graph on $4k+1$ vertices.
  If $G$ is self-complementary, then $G$ has exactly one vertex $v$ of degree $2k$, and $G-v$ is also self-complementary.
\end{proposition}
\begin{proof}
  Let $\sigma$ be an antimorphism of $G$.
  By Lemma~\ref{lem:length-of-circular-permutation}, there exists a cycle of length one in $\sigma$; let it be $(v)$.
  We can write $\sigma = \sigma_1 \dots \sigma_p (v)$.
  By Proposition~\ref{prop:every-subset-of-cycles-are-scs}, $G - v$ is a self-complementary with $\sigma = \sigma_1 \dots \sigma_p$ as an antimorphism.
  Since it is an induced subgraph of a split graph, it is a self-complementary split graph, and has a unique split partition $K\uplus I$ by Lemma~\ref{lem:unique-split-partition}.
  The degree of $v$ is $|K| = 2 k$.
  On the other hand, every vertex in $K$ has at least one neighbor in $I$: otherwise, we can move it from $K$ to $I$ to get another split partition of $G - v$.  Thus, $d(x) > 2k$ for each vertex $x\in K$.
  In a similar way, we can conclude that $d(x) < 2k$ for each vertex $x\in I$.
\end{proof}

A \emph{pseudo-split graph} is either a split graph, or a graph whose vertex set can be partitioned into a clique $K$, an independent set $I$, and a set $C$ that (1) induces a $C_{5}$; (2) is complete to $K$; and (3) is nonadjacent to $I$.
We say that $K\uplus I \uplus C$ is a \emph{pseudo-split partition} of the graph, where $C$ may or may not be empty.
If $C$ is empty, then $K\uplus I$ is a split partition of the graph.  Otherwise, the graph has a unique pseudo-split partition.
Similar to split graphs, the complement of a pseudo-split graph remains a pseudo-split graph.

\begin{proposition}\label{lem:split-pseudo-split}
  Let $G$ be a self-complementary pseudo-split graph with a pseudo-split partition $K \uplus I \uplus C$.  If $C\ne \emptyset$, then $G-C$ is a self-complementary split graph of an even order.
\end{proposition}
\begin{proof}
  Let $\sigma$ be an antimorphism of $G$.
  In both $G$ and its complement, the only $C_{5}$ is induced by $C$.
  Thus, $\sigma(C) = C$.
  Since $C$ is complete to $K$ and nonadjacent to $I$, it follows that $\sigma(K) = I$ and $\sigma(I) = K$.
  Thus, $G - C$ is a self-complementary graph.
  It is clearly a split graph and has an even order.
\end{proof}

In the rest of this section, we are exclusively concerned with partitions of the vertex set of a graph $G$ into four nonempty subsets.
A partition $\mathcal{P} = \{V_{1},V_{2},V_{3},V_{4}\}$ of $V(G)$ is a \emph{rectangle partition} if $V_{1}$ is complete to $V_{2}$ and nonadjacent to $V_{3}$, while $V_{4}$ is complete to $V_{3}$ and nonadjacent to $V_{2}$, or a \emph{diamond partition} if $V_{1}$ is complete to $V_{3}$ while $V_{2}$ is nonadjacent to $V_{4}$.
See Figure~\ref{fig:model}.
Trotignon \cite{trotignon-05-self-complementary-graphs} conjectured that every $C_{5}$-free self-complementary graph $G$ admits one of the two partitions.  
We prove Trotignon's conjecture on self-complementary split graphs.  

\begin{lemma} \label{lem:scs-diamond-partition}
  Every self-complementary split graph $G$ admits a diamond partition.
  If $G$ has an even order, then it admits a diamond partition that is self-complementary.
\end{lemma}
\begin{proof}
  Let $K\uplus I$ be a split partition of $G$.
  For any proper and nonempty subset $K'\subseteq K$ and proper and nonempty subset $I'\subseteq I$, the partition $K', K\setminus K', I', I\setminus I'$
  is a diamond partition.

  Now suppose that the order of $G$ is $4 k$.
  We fix an arbitrary antimorphism $\sigma = \sigma_{1}\sigma_{2}\cdots \sigma_{p}$ of $G$.
  We may assume without loss of generality that for all $i = 1, \ldots, p$, the first vertex in $\sigma_{i}$ is in $K$.
  For $j= 1, \ldots, |\sigma_{i}|$, we assign the $j$th vertex of $\sigma_{i}$ to $V_{j\pmod 4}$.  For $j = 1, \ldots, 4$, we have $\sigma(V_{j}) = V_{j+1\pmod 4}$.
  Moreover, $V_{1}\cup V_{3} = K$ and $V_{2}\cup V_{4} = I$.
  Thus, $\{V_{1}, V_{2}, V_{3}, V_{4}\}$ is a self-complementary diamond partition of $G$.
\end{proof}

For a positive integer $k$, let $Z_{k}$ denote the graph obtained from a $P_{4}$ as follows.  We substitute each degree-one vertex with an independent set of $k$ vertices, and each degree-two vertex with a clique of $k$ vertices.

\begin{lemma}\label{lem:scs-rectangle-partition}
   A self-complementary split graph has a rectangle partition if and only if it is isomorphic to $Z_{k}$.
\end{lemma}
\begin{proof}
  The sufficiency is trivial, and we consider the necessity.
  Suppose that $G$ is a self-complementary split graph and it has a rectangle partition $\{V_{1}, V_{2}, V_{3}, V_{4}\}$.
  Let $K\uplus I$ be a split partition of $G$.
  There are at least one edge and at least one missing edge between any three parts.  Thus, vertices in $K$ are assigned to precisely two parts in the partition.
  By the definition of rectangle partition, $K$ is either $V_{2}\cup V_{3}$ or $V_{1}\cup V_{4}$.
  Assume without loss of generality that $K = V_{2}\cup V_{3}$.
  Since $V_{2}$ is complete to $V_{1}$ and nonadjacent to $V_{4}$, any antimorphism of $G$ maps $V_{2}$ to either $V_{1}$ or $V_{4}$.
  If $|V_{2}| \ne |V_{3}|$, then the numbers of edges between $K$ and $I$ in $G$ and $\overline G$ are different.  This is impossible.
\end{proof}

\section{Forcibly self-complementary degree sequences}
\label{sec:degree-sequences}

The \emph{degree sequence}\index{degree sequence} of a graph $G$ is the sequence of degrees of all vertices, listed in non-increasing order, and $G$ is a \emph{realization} of this degree sequence.
For our purpose, it is more convenient to use a compact form of degree sequences where the same degrees are grouped:
\[
  \left(d_i^{n_i}\right)_{i=1}^\ell = 
  \left(d_{1}^{n_{1}}, \dots, d_{\ell}^{n_{\ell}}\right) =
  \left( \underbrace{d_{1}, \ldots, d_{1}}_{n_1}, \underbrace{d_{2}, \ldots, d_{2}}_{n_2},  \ldots, \underbrace{d_{\ell}, \ldots, d_{\ell}}_{n_\ell}\right).
\]
Note that we always have $d_1 > d_2 > \cdots > d_\ell$.  
For example, the degree sequences of the first two graphs in Figure~\ref{fig:split-sc-8-vertices} are both
\[
  \left(5^4, 2^4\right) = (5,5,5,5,2,2,2,2).
\]
These two graphs are not isomorphic; thus, a degree sequence may have non-isomorphic realizations. 
% A sequence of integers is \emph{graphical} if it is the degree sequence of some graph.

For four vertices $v_{1}$, $v_{2}$, $v_{3}$, and $v_{4}$ such that $v_{1}$ is adjacent to $v_{2}$ but not to $v_{3}$ while $v_{4}$ is adjacent to $v_{3}$ but not to $v_{2}$, 
the operation of replacing $v_{1} v_{2}$ and $v_{3} v_{4}$ with $v_{1} v_{3}$ and $v_{2} v_{4}$ is a \emph{2-switch}, denoted as
\[
  (v_{1} v_{2}, v_{3} v_{4}) \rightarrow (v_{1} v_{3}, v_{2} v_{4}).
\]
See Figure~\ref{fig:2-switches}.
It is easy to check that this operation does not change the degree of any vertex.
Indeed, it is well known that any two graphs of the same degree sequence can be transformed into each other by 2-switches~\cite{ryser-57-binary-matrices}.

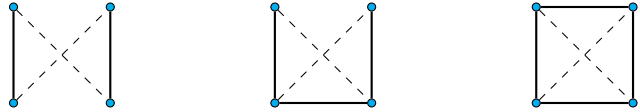
\begin{figure}[h]
  \centering
  \begin{subfigure}[b]{0.2\textwidth}
    \centering
    \begin{tikzpicture}[scale=.6]
      \def\n{4}
      \def\radius{1.5}
      \foreach \i in {1,..., \n} {
        \pgfmathsetmacro{\angle}{ (\i + .5) * (360 / \n)}
        \node[filled vertex] (v\i) at (\angle:\radius) {};
      }
      \draw[thick] (v1) -- (v2) (v3) -- (v4);
      \draw[dashed] (v1) -- (v3) (v2) -- (v4);
    \end{tikzpicture}
  \end{subfigure}
  \,
  \begin{subfigure}[b]{0.2\textwidth}
    \centering
    \begin{tikzpicture}[scale=.6]
      \def\n{4}
      \def\radius{1.5}
      \foreach \i in {1,..., \n} {
        \pgfmathsetmacro{\angle}{ (\i + .5) * (360 / \n)}
        \node[filled vertex] (v\i) at (\angle:\radius) {};
      }
      \draw[thick] (v1) -- (v2) (v3) -- (v4);
      \draw[thick] (v2) -- (v3);      
      \draw[dashed] (v1) -- (v3) (v2) -- (v4);
    \end{tikzpicture}
  \end{subfigure}
  \,
  \begin{subfigure}[b]{0.2\textwidth}
    \centering
    \begin{tikzpicture}[scale=.6]
      \def\n{4}
      \def\radius{1.5}
      \foreach \i in {1,..., \n} {
        \pgfmathsetmacro{\angle}{ (\i + .5) * (360 / \n)}
        \node[filled vertex] (v\i) at (\angle:\radius) {};
      }
      \draw[thick] (v1) -- (v2) (v3) -- (v4);
      \draw[thick] (v2) -- (v3) (v1) -- (v4);      
      \draw[dashed] (v1) -- (v3) (v2) -- (v4);
    \end{tikzpicture}
  \end{subfigure}
  \caption{Illustrations for 2-switches.}
  \label{fig:2-switches}
\end{figure}

\begin{lemma}[\cite{ryser-57-binary-matrices}]
  \label{lem:2-switches}
  Two graphs have the same degree sequence if and only if they can be transformed into each other by a series of 2-switches.
\end{lemma}

The subgraph induced by the four vertices involved in a 2-switch operation must be a $2 K_{2}$, $P_{4}$, or $C_{4}$.  Moreover, after the operation, the four vertices induce an isomorphic subgraph.
Since a split graph $G$ cannot contain any $2 K_{2}$ or $C_{4}$~\cite{foldes-77-split-graphs}, a 2-switch must be done on a $P_{4}$.
In any split partition $K\uplus I$ of $G$, the two degree-one vertices of $P_{4}$ are from $I$, while the others from $K$.
The graph remains a split graph after this operation.
Thus, if a degree sequence has a realization that is a split graph, then all its realizations are split graphs \cite{foldes-77-split-graphs}.  A similar statement holds for pseudo-split graphs \cite{maffray-94-pseudo-split}.

We do not have a similar claim on degree sequences of self-complementary graphs. 
Clapham and Kleitman~\cite{clapham-76-self-complementary-degree-sequences} have fully characterized all such degree sequences, called \emph{potentially self-complementary degree sequences}.
A degree sequence is \emph{forcibly self-complementary} if all of its realizations are self-complementary.

\begin{proposition}\label{lem:basic-graphs}
  The following degree sequences are all forcibly self-complementary: $(0^1)$, $(2^2, 1^2)$, $(2^5)$, and $(5^4, 2^4)$.
\end{proposition}
\begin{proof}
  It is trivial for $(0^1)$.
  Applying a 2-switch operation to a realization of $(2^2, 1^2)$ or $(2^5)$ leads to an isomorphic graph.
  A 2-switch operation transforms the graph in Figure~\ref{fig:split-sc-8-vertices}(a) into Figure~\ref{fig:split-sc-8-vertices}(b), and vice versa.
  Thus, the statement follows from Lemma~\ref{lem:2-switches}.
\end{proof}

We take $p$ vertex-disjoint graphs $S_{1}$, $S_{2}$, $\ldots$, $S_{p}$, each of which is isomorphic to $P_{4}$, or one of the graphs in Figure~\ref{fig:share-degree-sequence}(a, b).
For $i = 1, \ldots, p$, let $K_{i}\uplus I_{i}$ denote the unique split partition of $S_{i}$.
Let $C$ be another set of $0$, $1$, or $5$ vertices.
We add all possible edges among $\bigcup_{i=1}^{p} K_{i}$ to make it a clique, and for each $i = 1, \ldots, p$, add all possible edges between $K_{i}$ and $\bigcup_{j=i+1}^{p} I_{j}$.\footnote{The reader familiar with threshold graphs may note its use here.
If we contract $K_{i}$ and $I_{i}$ into two vertices, the graph we constructed is a threshold graph.
  Threshold graphs have a stronger characterization by degree sequences.  Since a threshold graph free of $2K_{2}$, $P_{4}$, and $C_{4}$, no 2-switch is possible on it.
 Thus, the degree sequence of a threshold graph has a unique realization.
 % Recall that a threshold graph has a threshold partition $K \uplus I$ with clique $K$ and independent set $I$ such that $\{N(v)\mid v\in I \}$ can be linearly ordered by subset relation.  We may assume that there exists a vertex $v$ in $I$ that is adjacent to all the vertices in $K$; otherwise, there is a vertex $v\in K$ that is not adjacent to any vertex in $I$, and we can use another threshold partition $(K\setminus \{v\})\uplus (I\cup \{v\})$.  It is easy to see that $d(x) \ge |C| \ge d(y)$ for every $x\in K$ and for every $y\in I$.
}
Finally, we add all possible edges between $C$ and $\bigcup_{i=1}^{p} K_{i}$,  and add edges to make $C$ a cycle if $|C| = 5$.
Let $\mathcal{E}$ denote the set of graphs that can be constructed as above.

\begin{lemma}\label{lem:elementary-graph}
  All graphs in $\mathcal{E}$ are self-complementary pseudo-split graphs, and their degree sequences are forcibly self-complementary.
\end{lemma}
\begin{proof}
  Let $G$ be any graph in $\mathcal{E}$. 
  It has a split partition $(\bigcup_{i=1}^{p} K_{i} \cup C)\uplus \bigcup_{i=1}^{p} I_{i}$ when $|C| \le 1$, and a pseudo-split partition $\bigcup_{i=1}^{p} K_{i} \uplus \bigcup_{i=1}^{p} I_{i}\uplus C$ otherwise.
  To show that it is self-complementary, we construct an antimorphism $\sigma$ for it.  For each $i = 1, \ldots, p$, we take an antimorphism $\sigma_{i}$ of $S_{i}$, and set $\sigma(x) = \sigma_{i}(x)$ for all $x\in V(S_{i})$.
  If $C$ consists of a single vertex $v$, we set $\sigma(v) = v$.
  If $|C| = 5$, we take an antimorphism $\sigma_{p+1}$ of $C_{5}$ and set $\sigma(x) = \sigma_{p+1}(x)$ for all $x\in C$.
  It is easy to verify that a pair of vertices $u, v$ are adjacent in $G$ if and only if $\sigma(u)$ and $\sigma(v)$ are adjacent in $\overline G$.

  For the second assertion, we show that applying a 2-switch to a graph $G$ in $\mathcal{E}$ leads to another graph in $\mathcal{E}$.
  Since $G$ is a split graph, a 2-switch can only be applied to a $P_4$.
  For two vertices $v_{1}\in K_{i}$ and $v_{2}\in K_{j}$ with $i < j$, we have $N[v_{2}]\subseteq N[v_{1}]$.  Thus, there cannot be any $P_{4}$ involving both $v_{1}$ and $v_{2}$.
  A similar argument applies to two vertices in $I_{i}$ and $I_{j}$ with $i\ne j$.
  Therefore, a 2-switch can be applied either \textit{inside} $C$ or \textit{inside} $S_{i}$ for some $i \in \{1, \ldots, p\}$.
  By Proposition~\ref{lem:basic-graphs}, the resulting graph is in $\mathcal{E}$.
\end{proof}

We refer to graphs in $\mathcal{E}$ as \emph{elementary self-complementary pseudo-split graphs}.
The rest of this section is devoted to showing that all realizations of forcibly self-complementary degree sequences are elementary self-complementary pseudo-split graphs. We start with a simple observation on potentially self-complementary degree sequences with two different degrees. It can be derived from Clapham and Kleitman~\cite{clapham-76-self-complementary-degree-sequences}. We provide a direct and simple proof here.

\begin{proposition}
  \label{lem:homogeneous-construction}
  There is a self-complementary graph of the degree sequence $(d^{2k}, (4k - 1 - d)^{2k})$ if and only if $2k \leq d \leq 3k - 1$.
  Moreover, there exists a self-complementary graph with a one-cycle antimorphism.
\end{proposition}
\begin{proof}
  Necessity. By the definition of degree sequences, $d > 4k - 1 - d$. Therefore, $d \geq 2k$. Let $H$ be the set of vertices of degree $d$ and $L$ the set of vertices of degree $4k - 1 - d$. Each vertex in $H$ has at most $|H| - 1 = 2 k - 1$ neighbors in $H$.  Thus, the number of edges between $H$ and $L$ is at least
  $2 k (d - 2 k + 1) $.
  On the other hand, the number of edges between $H$ and $L$ is at most
  $2 k (4k - 1 - d)$.
  Thus, $4k - 1 - d \ge d - 2 k + 1$, and the claim follows.

  Sufficiency.
  We construct a self-complementary graph that has an antimorphism with exactly one cycle $(v_1 v_2 \cdots, v_{4k})$ by using the method of Gibbs~\cite{gibbs-74-self-complementary}. Note that the adjacencies between the first vertex and the other vertices decide the graph. We set the neighborhood of $v_{1}$ to be $\{v_{2}, v_{6}, \ldots, v_{4 k - 2}\}\cup X$, where
  \[
    X = \begin{cases}
          \{v_{3}, v_{5}, \ldots, v_{d - k}\} \cup \{v_{2k+1}\} \cup \{v_{4k-1}, v_{4k-3}, \ldots, v_{5k - d + 2}\}, & d \not\equiv k \pmod 2
          \\
          \{v_{3}, v_{5}, \ldots, v_{d - k + 1}\} \cup \{v_{4k-1}, v_{4k-3}, \ldots, v_{5k - d + 1}\}, & d \equiv k \pmod 2
        \end{cases}
  \]
  In the constructed graph, all odd-number vertices have degree $d$, and the others $4 k - d - 1$.
\end{proof}

The next proposition considers the parity of the number of vertices with a specific degree. It directly follows from Clapham and Kleitman~\cite{clapham-76-self-complementary-degree-sequences}, and Xu and Wong~\cite[Theorem 4.4]{xu-00-self-complementary-graphs}.

\begin{proposition}[\cite{clapham-76-self-complementary-degree-sequences, xu-00-self-complementary-graphs}]\label{prop:even-number}
  Let $G$ be a graph of order $4k$ and $v$ an arbitrary vertex of $G$. Let $H$ and $L$ be the $2k$ vertices of the largest and smallest degrees, respectively in $G$. If $G$ is self-complementary, then there exist an even number of vertices with degree $d_{G}(v)$ in $G$, an even number of vertices with degree $d_{G[H]}(v)$ in $G[H]$, and an even number of vertices with degree $d_{G[L]}(v)$ in $G[L]$.
\end{proposition}

In general, it is quite challenging to verify that a degree sequence is indeed forcibly self-complementary.
On the other hand, to show that a degree sequence is not forcibly self-complementary, it suffices to construct a realization that is not self-complementary.
We have seen that degree sequences $(0^{1})$, $(2^{5})$ and $(2^{2}, 1^{2})$, $(5^{4}, 2^{4})$ are forcibly self-complementary.
They are the only ones of these forms.

\begin{figure}[h]
  \centering
  \begin{tikzpicture}[scale=1.3]
    \def\k{2}
    \pgfmathsetmacro{\n}{4 * \k + 1}
    \coordinate (v0) at ({90}:1) {};
    \foreach \i in {0, 1,..., \n} {
      \pgfmathsetmacro{\angle}{90 - \i * (360 / \n)}
      \node[filled vertex] (v\i) at (\angle:1) {};
    }
    \foreach \i in {1,..., \n} {
      \foreach \j in {1,..., \k} {
        \pgfmathsetmacro{\e}{int(mod(\i + \j, \n))}
        \draw (v\i) -- (v\e);
      }
    }
    
  \end{tikzpicture}
  \caption{The graph $C_{9}^{2}$, with degree sequence $(4^{9})$, is not self-complementary.}
  \label{fig:reguar-not-sc}
\end{figure}
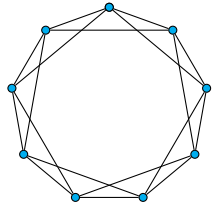

\begin{proposition}\label{prop:ds-not-forcibly}
  The following degree sequences are not forcibly self-complementary.
  \begin{enumerate}[i)]
  \item $((2k)^{4k+1})$, where $k \geq 2$.
  % \item $((3k-1)^{2k},(k)^{2k})$, where $k \ge 3$;
  \item $(d^{2k},(n-1-d)^{2k})$, where $k \ge 2$ and $d \neq 5$.
  \item $(d^{2k_{1}}, (d-1)^{2k_{2}}, (n-d)^{2k_{2}}, (n-1-d)^{2k_{1}})$, where $k_{1}, k_{2} > 0$.
  \end{enumerate}  
\end{proposition}
\begin{proof}
  The statement holds vacuously if the degree sequence is not potentially self-complementary.  Henceforth, we assume that they are.
  
  (i)
    We start from a cycle graph on $4k+1$ vertices, and add an edge between every pair of vertices with distance at most $k$ on this cycle.
    The resulting graph is denoted as $C_{4k+1}^{k}$.
    As an example, the graph for $k = 2$ is in Figure~\ref{fig:reguar-not-sc}.
    To see that the graph $C_{4k+1}^{k}$ is not self-complementary, note that for any vertex $v$, there are $3k(k-1)/2$ edges among $N(v)$ and $k(k-1)/2$ missing edges among $V(G)\setminus N[v]$.
  
    (ii) By Proposition~\ref{lem:homogeneous-construction}, we have that $2 k \leq d \leq 3 k - 1$.
    The graph in Figure~\ref{fig:an-example-not-forcibly} has degree sequence $(4^4,3^4)$ and is not self-complementary.  In the rest, $k \geq 3$.

    \begin{figure}[b]
      \centering
      \begin{tikzpicture}
          \draw (0,2) -- (0,1) -- (0,0) -- (1,0) -- (2,0) -- (2,1) -- (2,2) -- (1,2);
          \draw (2,1) -- (1,2) -- (0,1) -- (1,0);
          \draw (0,2) -- (2,0);
          \draw (1,0) -- (1,2);
          \draw (0,2) -- (2,1);
          \draw (0,0) -- (2,2);
          \foreach \x in {0,2} {
            \foreach \y in {0,2} {
              \node[empty vertex] () at (\x,\y) {};
            }
          }
          \foreach \list[count=\y from 0] in {{1},{0,2},{1}}
          \foreach \x in \list {
            \node[filled vertex] () at (\x,\y) {};
          }
        \end{tikzpicture}
      \caption{A graph, with degree sequences $(4^4,3^4)$, is not self-complementary.}
      \label{fig:an-example-not-forcibly}
    \end{figure}
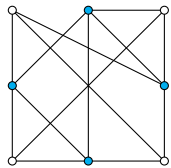
    
    Case 1: $d = 3 k - 1$. Starting with a $P_{4}$, we substitute each degree-one vertex with an independent set of $k$ vertices, and each degree-two vertex with a clique of $k$ vertices.
    The degree sequence is $((3k-1)^{2k},(k)^{2k})$.
    We label the vertices of degree $3k-1$ as $u_1, \ldots, u_{2k}$ and vertices of degree $k$ as $v_1, \ldots, v_{2k}$.
    For $i = 1, \ldots, k$, we conduct $(u_k v_{i}, u_{k+i} v_{k+i})\rightarrow (u_{k} v_{k+i}, u_{k+i} v_i)$.
    See Figure~\ref{fig:12-vertices} for the example of $k = 3$.  We show that the resulting graph is not self-complementary.  Note that the $k-1$ vertices $u_1, \ldots, u_{k-1}$ are twins (having the same neighborhood). It suffices to argue that there are no twins in $v_1, \ldots, v_{2k}$. Since $N(u_k) = \{v_{k+1}, \ldots, v_{2k}\}$, we separate them into $v_1, \ldots, v_{k}$ and $v_{k+1}, \ldots, v_{2k}$. For $1 \le i < j \le k$, vertices $v_i$ and $v_j$ are not twins because $u_{k+i}$ is adjacent to $v_i$ but not $v_j$. For $k+1 \le i < j \le 2k$, vertices $v_i$ and $v_j$ are not twins because $u_i$ is adjacent to $v_j$ but not $v_i$.
  
    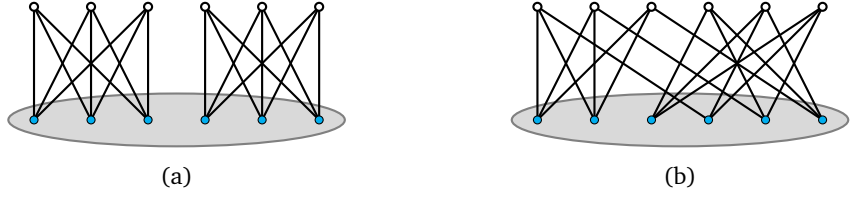
\begin{figure}[t]
    \centering
    \tikzset{every path/.style={thick}}
    \begin{subfigure}[b]{.4\textwidth}
      \centering
      \begin{tikzpicture}[xscale=.75]
        \draw[fill = gray!30, draw = gray] (3.5, 0) ellipse (2.95 and 0.35);
        \foreach \i in {1,..., 3}
            \foreach \j in {1,..., 3}
                \draw (\i, 1.5) -- (\j, 0);
        \foreach \i in {4,..., 6}
            \foreach \j in {4,..., 6}
                \draw (\i, 1.5) -- (\j, 0);
        \foreach \i in {1, ..., 6} {
          \node[filled vertex] (c\i) at (\i, 0) {};
          \node[empty vertex] (i\i) at (\i, 1.5) {};
        }
      \end{tikzpicture}
      \caption{}
      \label{fig:split-sc-12-vertices}
    \end{subfigure}
    \,
    \begin{subfigure}[b]{.4\textwidth}
      \centering
      \begin{tikzpicture}[xscale=.75]
        \draw[fill = gray!30, draw = gray] (3.5, 0) ellipse (2.95 and 0.35);
        \foreach \i in {1,..., 3}
        \foreach \j in {1,..., 2}
        \draw (\i, 1.5) -- (\j, 0);
        \foreach \i in {4,..., 6} {
          \draw (\i, 1.5) -- (3, 0);
          \foreach \j in {4,..., 6}
          \ifthenelse{\i=\j}{\draw (\i-3, 1.5) -- (\i, 0);}{\draw (\i, 1.5) -- (\j, 0);};        
          }
        \foreach \i in {1, ..., 6} {
          \node[filled vertex] (c\i) at (\i, 0) {};
          \node[empty vertex] (i\i) at (\i, 1.5) {};
        }
      \end{tikzpicture}
      \caption{}
      \label{}
    \end{subfigure}
    \caption{Two graphs with degree sequence $(8^{6}, 3^{6})$, where
      (a) is self-complementary but (b) not.}
    \label{fig:12-vertices}
  \end{figure}
    
  Case 2: $d < 3 k - 1$.  Using the method shown in Proposition~\ref{lem:homogeneous-construction}, we can construct a realization $G$ of $(d^{2k},(n-1-d)^{2k})$.
  Note that $G$ is self-complementary with an antimorphism $\sigma = (v_1 v_2 \cdots, v_{4k})$. Let $H = \{v_1, v_3, v_5, v_7, \dots, v_{4k-1}\}$. Note that the vertices in $H$ share the same degree $d$. 

  If $v_{1}$ is adjacent to $v_{2k+1}$, then it is not adjacent to $v_{2k-1}$; otherwise, from our construction, $v_{2k-1}$ must be $v_{d-k}$ and it implies that $d = 3k - 1$, a contradiction. The fact that $v_{1}$ is adjacent to $v_{2}$ implies that $v_{2k-1}$ is adjacent to $v_{2k}$ and  $v_{2k}$ is not adjacent to $v_{2k+1}$.  We conduct the 2-switch $(v_1v_{2k+1},v_{2k-1}v_{2k}) \rightarrow (v_1v_{2k-1},v_{2k}v_{2k+1})$, and denote by $G'$ the resulting graph.
  It can be observed that
  \[
    |N_{G'}(v) \cap H| = 
    \begin{cases}
      |N_{G}(v) \cap H| + 1 & \text{if } v=v_{2k-1}, \\
      |N_{G}(v) \cap H| - 1 & \text{if } v=v_{2k+1} \text{, and} \\
      |N_{G}(v) \cap H|     & \text{if } v \in H\setminus \{2k-1, 2k+1\}. \\
     \end{cases}
\]
  The graph $G'$ is not self-complementary by Proposition~\ref{prop:even-number}.

  We now consider the case that $v_{1}$ is not adjacent to $v_{2k+1}$. From our construction, we know that $d-k$ is even and $v_{1}$ is adjacent to $v_{d-k+1}$ and not adjacent to $v_{d-k+3}$. The fact that $v_{1}$ is adjacent to $v_{2}$ and not adjacent to $v_{4}$ implies $v_{d-k+3}$ is adjacent to  $v_{d-k+4}$ and $v_{d-k+1}$ is not adjacent to $v_{d-k+4}$. By conducting the 2-switch $(v_1v_{d-k+1},v_{d-k+3}v_{d-k+4}) \rightarrow (v_1v_{d-k+3},v_{d-k+1}v_{d-k+4})$, the resulting graph $G'$ have the same degree sequence as $G$. By using arguments similar to the previous paragraph, it can be shown that $G'$ is not self-complementary.
  
    (iii)  We use $\tau$ to denote the degree sequence $(d^{2k_{1}}, (d-1)^{2k_{2}}, (n-d)^{2k_{2}}, (n-1-d)^{2k_{1}})$. 
    Since $\tau$ is potentially self-complementary, the inequality 
    \[
        k_1d + k_2(d-1) \leq (k_1 + k_2) (n - 1 - (k_1 + k_2))
    \]
      should be satisfied by the theorem in~\cite{clapham-76-potentially-self-complementary-sequences}. Therefore,
  \[
      d \leq n - 1 - (k_1 + k_2) + \frac{k_2}{k_1 + k_2} < n - 1 - (k_1 + k_2).
  \]
  By using the same theorem, it can be seen that the integer sequence $(d^{2k_{1}+ 2k_{2}}, (n-1-d)^{2k_{1}+ 2k_{2}})$ is potentially self-complementary. 
  
  Let $k = k_1 + k_2$. 
  We can construct a realization $G$ of $(d^{2k}, (n-1-d)^{2k})$ by using the method shown in Proposition~\ref{lem:homogeneous-construction}. 
  Note that $G$ is self-complementary with an antimorphism $\sigma = (v_1 v_2 \cdots, v_{4k})$ and all odd-numbered vertices have degree $d$, and the others have degree $4k - d - 1$. The fact that $v_1$ is adjacent to $v_3$ implies $\sigma^{4i}(v_1)$ is adjacent to $\sigma^{4i}(v_3)$ for all $i = 1, 2, \dots, k - 1$. Furthermore, since $v_1$ is adjacent to $v_2$, the vertex $v_3$ is adjacent to $v_4$ and $v_4$ is not adjacent to $v_5$. Moreover, we can further deduce that $\{v_5\}$ is complete to $\{v_{2}, v_{6}, \ldots, v_{4 k - 2}\}$ since $\{v_1\}$ is complete to $\{v_{2}, v_{6}, \ldots, v_{4 k - 2}\}$.
  
  We claim that $v_{1}$ is adjacent to $v_{5}$ in $G$. Suppose $v_{1}$ is not adjacent to $v_{5}$. Then $v_1$ is only adjacent to $v_3$ and $v_{4k-1}$ in $\{v_3, v_5, v_7, \dots, v_{4k-1}\}$. Since $d > n - 1 - d$, we have that $n$ can only be eight and the degree sequence of $G$ is $(4^4,3^4)$. Note that $d > d-1 > n-2 > n - 1 - d$. The difference between $d$ and $n-1-d$ is at least three. We encounter a contradiction.
  
  We now remove the edge $\sigma^{4i}(v_1)\sigma^{4i}(v_3)$ and add edge $\sigma^{4i + 1}(v_1)\sigma^{4i + 1}(v_3)$ for all $i = 0, 1, 2, \dots, k_2-1$. The resulting graph $G'$ is a realization of the degree sequence $\tau$. In $G'$, the vertex $v_1$ is adjacent to $v_5$ and not adjacent to $v_3$. The vertex $v_4$ is adjacent to $v_3$ and not adjacent to $v_5$. By conducting the 2-switch $(v_1v_5,v_3v_4) \rightarrow (v_1v_3,v_4v_5)$, the resulting graph $G''$ have the same degree sequence as $G'$. 
  
  We show that $G''$ is not self-complementary. Let $H = \{v_1, v_3, v_5, v_7, \dots, v_{4k-1}\}$ and $L = \{v_2, v_4, v_6, \dots, v_{4k}\}$. Suppose $G''$ is a self-complementary graph. Then any antimorphism $\sigma'$ of $G''$ maps $H$ to $L$ and vice versa. Since $v_5$ is adjacent to $v_4$ and $\{v_5\}$ is complete to $\{v_{2}, v_{6}, \ldots, v_{4 k - 2}\}$, the vertex $v_5$ has $k + 1$ neighbors in $L$. Therefore, $\sigma'(v_5)$ is in $L$ and it has $k + 1$ non-neighbors in $H$. Every vertex in $L$ has $k$ neighbors in $H$ and $|H| = 2k$. No vertex in $L$ can have $k + 1$ non-neighbors in $H$. We encounter a contradiction.
\end{proof}

Let $G$ be a self-complementary graph with $\ell$ different degrees $d_{1}$, $\ldots$, $d_{\ell}$. For each $i = 1, \ldots, \ell$, let $V_{i}(G) = \{v\in V(G)\mid d(v) = d_{i}\}$, and we define the \emph{$i$th slice} of $G$ as the induced subgraph $S_{i}(G) = G[V_{i}\cup V_{\ell + 1 - i}]$. We may drop $(G)$ when the graph is clear from the context.
Note that $V_{i} = V_{\ell + 1 - i}$ and $S_{i} = G[V_{i}]$ when $\ell$ is odd and $i = (\ell + 1)/2$.
Each slice must be self-complementary, and more importantly, its degree sequence is forcibly self-complementary.

\begin{lemma} \label{lem:sub-graph-forcibly}
  Let $\tau$ be a forcibly self-complementary degree sequence, $G$ a realization of $\tau$, and $\sigma$ an antimorphism of $G$.
  The degree sequence of every slice of $G$ is forcibly self-complementary.
\end{lemma}
\begin{proof}
  Let $\tau = (d_i^{n_i})_{i=1}^\ell$. 
  Since $d_{1} > d_{2} > \cdots > d_{\ell}$, the antimorphism $\sigma$ maps the vertices from $V_{i}$ to $V_{\ell + 1 - i}$, and vice versa.
  Therefore, $n_i = n_{\ell + 1 - i}$, and
  the cycles of $\sigma$ consisting of vertices from $V_{i}\cup V_{\ell + 1 - i}$ is an antimorphism of $S_{i}$.
  Therefore, $S_{i}$ is self-complementary.

  We now verify any graph $S$ with the same degree sequence as $S_{i}$ is self-complementary.
  By Lemma~\ref{lem:2-switches}, 
  we can transform $S_{i}$ to $S$ by a sequence of 2-switches applied on vertices in $V_{i}\cup V_{\ell + 1 - i}$.
  We can apply the same sequence of 2-switches to $G$, which lead to a graph $G'$ with degree sequence $\tau$.
  Note that $S$ is the $i$th slice of $G'$, hence self-complementary.
\end{proof}

The following result follows from Lemma~\ref{lem:sub-graph-forcibly}.
\begin{corollary}\label{lem:no-2-switch}
  Let $G$ be a graph with $\ell$ different degrees. If the degree sequence of $G$ is forcibly self-complementary, then there cannot be a 2-switch that changes the number of edges in $S_{i}$ or between $V_{i}$ and $V_{\ell + 1 - i}$ for every $i \in \{1, \dots, \ell\} \setminus \{\frac{\ell + 1}{2}\}$.
\end{corollary}
\begin{proof}
  Since $S_{i}(G)$ is a self-complementary graph, the number of edges is fixed.  Since there exists an antimorphism of $S_{i}(G)$ that maps $V_{i}(G)$ to $V_{\ell + 1 - i}(G)$, the number of edges between them is fixed. Suppose there exists a 2-switch that changes the number of edges in $S_{i}(G)$ or between $V_{i}(G)$ and $V_{\ell + 1 - i}(G)$ for some $i \in \{1, \dots, \ell\} \setminus \{\frac{\ell + 1}{2}\}$. Let $G'$ be the resulting graph. Consequently, $S_{i}(G')$ is not self-complementary. Since the 2-switch operation does not change the degree of any vertex, the degree sequence of $S_{i}(G')$ should be forcibly self-complementary by Lemma~\ref{lem:sub-graph-forcibly}. We encounter a contradiction.
\end{proof}

Two vertices in $V_i$ share the same degree \textit{in the $i$th slice}.

\begin{lemma}\label{lem:homogeneous}
  Let $G$ be a graph with $\ell$ different degrees. If the degree sequence of $G$ is forcibly self-complementary, then for each $i \in \{1, \ldots, \ell\}$, the vertices in $V_{i}$ share the same degree in $S_{i}$.
\end{lemma}
\begin{proof}
  Suppose for contradiction that vertices in $V_{i}$ have different degrees in $S_{i}(G)$.
  By Lemma~\ref{lem:sub-graph-forcibly}, the degree sequence of $S_{i}(G)$ is forcibly self-complementary.
  It cannot be of the form $(d^{2k_{1}}, (d-1)^{2k_{2}}, (n-d)^{2k_{2}}, (n-1-d)^{2k_{1}})$ by Proposition~\ref{prop:ds-not-forcibly}(iii).
  Thus, there must be two vertices $v_{1}$ and $v_{2}$ in $V_{i}$ such that
 \[
   d = d_{S_{i}(G)}(v_{1}) > d_{S_{i}(G)}(v_{2}) + 1.
 \]
 There exists a vertex
 \[
   x_{1}\in V(S_{i}(G))\cap N(v_{1})\setminus N(v_{2}). %
 \]
 On the other hand, since $d_{G}(v_{1}) = d_{G}(v_{2})$, there must be a vertex
 \[
   x_{2} \in N(v_{2})\setminus \left( N(v_{1})\cup V(S_{i}(G)) \right).
 \]
 We apply the 2-switch $(x_{1} v_{1}, x_{2} v_{2})\rightarrow (x_{1} v_{2}, x_{2} v_{1})$ to $G$ and denote by $G'$ the resulting graph.
 By assumption, $G'$ is also self-complementary.
 By Lemma~\ref{lem:sub-graph-forcibly}, $S_{i}(G)$ is self-complementary, and hence there are an even number of vertices with degree $d$ by Proposition~\ref{prop:even-number}.
 The degree of a vertex $x$ in $S_{i}(G')$ is
 \[
   \begin{cases}
     d_{S_{i}(G)}(x) - 1 & x = v_{1},
     \\
     d_{S_{i}(G)}(x) + 1 & x = v_{2},
     \\
     d_{S_{i}(G)}(x) & \text{otherwise}.
   \end{cases}
 \]
 Thus, the number of vertices with degree $d$ in $S_{i}(G')$ is odd.  Hence, $S_{i}(G')$ is not self-complementary by Proposition~\ref{prop:even-number}, which contradicts Lemma~\ref{lem:sub-graph-forcibly}.
\end{proof}

We next show all possible configurations for the slices of $G$. 
\begin{lemma}
\label{lem:focibly-ds-cycle-scs}
  Let $G$ be a graph with $\ell$ different degrees. If the degree sequence of $G$ is forcibly self-complementary, then  
  \begin{enumerate}[i)]
    \item $S_{i}$ is either a $P_{4}$ or one of the graphs in Figure~\ref{fig:share-degree-sequence}(a, b) for every $i \in \{1, \dots, \ell\} \setminus \{\frac{\ell + 1}{2}\}$, and 
    \item $S_{(\ell+1)/2}$ is either a $C_{5}$ or contains exactly one vertex if $\ell$ is odd.
  \end{enumerate}
\end{lemma}
\begin{proof}
    For all $i = 1, \ldots, \ell$, the induced subgraph $S_{i}$ of $G$ is self-complementary by Lemma~\ref{lem:sub-graph-forcibly}. Furthermore, $S_{i}$ is either a regular graph or has two different degrees (Lemma~\ref{lem:homogeneous}).
    By considering the number of edges in $S_{i}$, we can deduce that $S_{i}$ is a regular graph if and only if its order is odd. From the proof of Lemma~\ref{lem:sub-graph-forcibly}, we know that the order of $S_{i}$ is odd if and only if $\ell$ is odd and $i = (\ell+1)/2$. 
    
    If $\ell$ is odd, then $S_{(\ell+1)/2}$ is a regular graph. Let $|V(S_{i})| = 4k+1$ for some integer $k \geq 0$. It can be derived that the degree sequence of $S_{(\ell+1)/2}$ is $((2k)^{4k+1})$. By Lemma~\ref{lem:sub-graph-forcibly} and Proposition~\ref{prop:ds-not-forcibly}(i), we can obtain that $k \leq 1$ and the degree sequence of $S_{(\ell+1)/2}$ can either be $(2^5)$ or $(0^1)$. Therefore, $S_{(\ell+1)/2}$ is either a $C_{5}$ or contains exactly one vertex.

    For every $i \in \{1, \dots, \ell\} \setminus \{\frac{\ell + 1}{2}\}$, we may assume that $S_{i}$ has $4k$ vertices for some positive integer $k$ and the degree sequence of $S_{i}$ is $(d^{2k},(4k-1-d)^{2k})$ for some positive integer $d$. By Lemma~\ref{lem:sub-graph-forcibly} and Proposition~\ref{prop:ds-not-forcibly}(ii), we can obtain that $k = 1$ or $d = 5$. Since $S_{i}$ is a self-complementary graph, the degree sequence of $S_{i}$ can either be $(5^4,2^4)$ or $(2^2,1^2)$ by Proposition~\ref{lem:homogeneous-construction}. Consequently, $S_{i}$ is either a $P_{4}$ or one of the graphs in Figure~\ref{fig:share-degree-sequence}(a, b).
\end{proof}

By Lemma~\ref{lem:focibly-ds-cycle-scs}, the induced subgraph $S_i$ is a self-complementary split graph for every $i \in \{1,2,\dots, \lfloor \ell/2 \rfloor\}$, 
We use $K_{i} \uplus I_{i}$ to denote the unique split partition of $S_i$ (Lemma~\ref{lem:unique-split-partition}).
Moreover, no vertex in $I_{i}$ is adjacent to all the vertices in $K_{i}$.

\begin{lemma}\label{lem:homogeneous2}
  Let $G$ be a graph with $\ell$ different degrees. If the degree sequence of $G$ is forcibly self-complementary, then for every $i \in \{1,2,\dots, \lfloor \ell/2 \rfloor\}$, 
  \begin{enumerate}[i)]
      \item $V_{i}$ is a clique, and $V_{\ell + 1 - i}$ an independent set; and
      \item if a vertex in $V(G)\setminus V_{i}$ has a neighbor in $V_{\ell + 1 - i}$, then it is adjacent to all the vertices in $V_{i}\cup V_{\ell + 1 - i}$.
  \end{enumerate}
\end{lemma}
\begin{proof}
  (i)
  Suppose for contradiction that there is a vertex $v_{1}\in V_{i}\cap I_{i}$.
  By Lemma~\ref{lem:focibly-ds-cycle-scs}(i), $S_{i}$ is either a $P_{4}$ or one of the graphs in Figure~\ref{fig:share-degree-sequence}(a, b).
  We can find a vertex $v_{2}\in K_{i}\setminus N(v_{1})$ and a vertex $x_{2}\in N(v_{2})\cap I_{i}$.  Note that $x_{2}$ is not adjacent to $v_{1}$.
  Since $d_{G}(v_{1})> d_{G}(v_{2})$ while $d_{S_{i}}(v_{1})< d_{S_{i}}(v_{2})$, we can find a vertex $x_{1}$ in $V(G) \setminus V(S_i)$ that is adjacent to $v_{1}$ but not $v_{2}$.
  The applicability of 2-switch $(x_{1} v_{1}, x_{2} v_{2})\rightarrow (x_{1} v_{2}, x_{2} v_{1})$ violates Corollary~\ref{lem:no-2-switch}.

  (ii)
  Let $x_1\in V(G)\setminus V_{i}$ be adjacent to $v_{1}\in V_{\ell + 1 - i}$. Since $V_{\ell + 1 - i}$ is an independent set, it does not contain $x_{1}$.
  Suppose that there exists $v_{2}\in V_{\ell + 1 - i}\setminus N(x_1)$.
  Every vertex $x_{2} \in N(v_{2}) \cap V_{i}$ is adjacent to $v_{1}$. Otherwise, we may conduct the 2-switch $(x_{1} v_{1}, x_{2} v_{2})\rightarrow (x_{1} v_{2}, x_{2} v_{1})$, and denote by $G'$ the resulting graph. It can be seen that $S_{i}(G')$ is neither a $P_{4}$ nor one of the graphs in Figure~\ref{fig:share-degree-sequence}(a, b), contradicting Lemma~\ref{lem:focibly-ds-cycle-scs}(i). Therefore, $S_{i}(G)$ can only be the graph in Figure~\ref{fig:share-degree-sequence}(b). Let $x_{3}$ be a non-neighbor of $v_{1}$ in $V_{i}$ and $v_{3}$ a neighbor of $x_{3}$ in $V_{\ell + 1 - i}$.
  Note that neither $x_{2}v_{3}$ nor $x_{3}v_{1}$ is an edge.  We may either conduct the 2-switch $(x_{1} v_{3}, x_{2} v_{2})\rightarrow (x_{1} v_{2}, x_{2} v_{3})$ or $(x_{1} v_{1}, x_{3} v_{3})\rightarrow (x_{1} v_{3}, x_{3} v_{1})$ to $G$, depending on whether $x_{1}$ is adjacent to $v_{3}$, and denote by $G'$ the resulting graph. The $i$th slice of $G'$ contradicts Lemma~\ref{lem:focibly-ds-cycle-scs}(i).
  
  Suppose that there exists a vertex $v_{2}\in V_{i}\setminus N(x_1)$. The vertex $v_{2}$ is not adjacent to $v_{1}$; otherwise, the applicability of the 2-switch $(x_{1} x_{2}, v_{1} v_{2})\rightarrow (x_{1} v_{2}, x_{2} v_{1})$ where $x_{2}$ is a non-neighbor of $v_{2}$ in $V_{\ell + 1 - i}$ violates Corollary~\ref{lem:no-2-switch}. Let $x_{3}$ be a neighbor of $v_{2}$ in $V_{\ell + 1 - i}$. Note that $v_{1}$ is not adjacent to $x_{3}$. The applicability of the 2-switch $(x_{1} v_{1}, x_{3} v_{2})\rightarrow (x_{1} v_{2}, x_{3} v_{1})$ violates Corollary~\ref{lem:no-2-switch}.
  \end{proof}

We are now ready to prove the main lemma.

\begin{lemma}\label{lem:focibly-ds-G-scs}
  Any realization of a forcibly self-complementary degree sequence is an elementary self-complementary pseudo-split graph.
\end{lemma}
\begin{proof}
  Let $G$ be an arbitrary realization of a forcibly self-complementary degree sequence and $\sigma$ an antimorphism of $G$.
  Lemmas~\ref{lem:focibly-ds-cycle-scs}(i) and~\ref{lem:homogeneous2}(i) imply that $V_{i} = K_{i}$ and $V_{\ell + 1 - i} = I_{i}$ for every $i \in \{1,2,\dots, \lfloor \ell/2 \rfloor\}$.
  Let $i,j$ be two distinct indices in $\{1,2,\dots, \lfloor \ell/2 \rfloor\}$.
  We argue that there cannot be any edge between $K_{i}$ and $I_{j}$ if $i > j$. Suppose for contradiction that there exists $x \in K_{i}$ that is adjacent to $y \in I_{j}$ for some $i > j$. By Lemma~\ref{lem:homogeneous2}(ii), $x$ is adjacent to all the vertices in $S_{j}$. Consequently, $\sigma(x)$ is in $I_{i}$ and has no neighbor in $S_{j}$. Let $v_{1}$ be a vertex in $K_{j}$. Since $v_{1}$ is not adjacent to $\sigma(x)$, it has no neighbor in $I_{i}$ by Lemma~\ref{lem:homogeneous2}(ii). Note that $S_{i}$ is either a $P_{4}$ or one of the graphs in Figure~\ref{fig:share-degree-sequence}(a, b) and so does $S_{j}$. If we focus on the graph induced by $V(S_{i}) \cup V(S_{j})$, we can observe that
  \[
        d_{G[V(S_{i}) \cup V(S_{j})]}(v_1) < d_{G[V(S_{i}) \cup V(S_{j})]}(x).
  \]
  Since $d_{G}(v_1) > d_{G} (x)$, we can find a vertex $x_{1}$ in $V(G) \setminus (V(S_{i}) \cup V(S_{j}))$ that is adjacent to $v_{1}$ but not $x$. Let $v_{2}$ be a neighbor of $x$ in $I_{i}$. Note that $v_{2}$ is not adjacent to $v_{1}$. We can conduct the 2-switch $(x_1v_1,xv_2) \rightarrow (x_1x,v_1v_2)$, violating Corollary~\ref{lem:no-2-switch}. Therefore, $I_{i}$ is nonadjacent to $\bigcup_{p=i+1}^{\lfloor \ell/2 \rfloor} K_{p}$ for all $i = 1, \ldots, \lfloor \ell/2 \rfloor$. Since $\sigma(I_{i}) = K_{i}$ and $\sigma(\bigcup_{p=i+1}^{\lfloor \ell/2 \rfloor} K_{p}) = \bigcup_{p=i+1}^{\lfloor \ell/2 \rfloor} I_{p}$, we can obtain that $K_{i}$ is complete to $\bigcup_{p=i+1}^{\lfloor \ell/2 \rfloor} I_{p}$. Moreover, $K_{i}$ is complete to $\bigcup_{p=i+1}^{\lfloor \ell/2 \rfloor} K_{p}$ by Lemma~\ref{lem:homogeneous2}(ii), and hence $I_{i}$ is nonadjacent to $\bigcup_{p=i+1}^{\lfloor \ell/2 \rfloor} I_{p}$.
  
  We are done if $\ell$ is even. In the rest, we assume that $\ell$ is odd. By Lemma~\ref{lem:focibly-ds-cycle-scs}(ii), the induced subgraph $S_{(\ell + 1)/2}$ is either a $C_{5}$ or contains exactly one vertex. It suffices to show that $V_{(\ell + 1)/2}$ is complete to $K_{i}$ and nonadjacent to $I_{i}$ for every $i \in \{1,2,\dots, \lfloor \ell/2 \rfloor\}$.
  Suppose $\sigma(v) = v$.
  When $V_{(\ell + 1)/2} = \{v\}$, the claim follows from Lemma~\ref{lem:homogeneous2} and that $\sigma(v) = v$ and $\sigma(V_{i}) = V_{\ell+1-i}$.
  Now $|V_{(\ell + 1)/2}| = 5$.
  Suppose for contradiction that there is a pair of adjacent vertices $v_{1} \in V_{(\ell + 1)/2}$ and $x \in I_{i}$.
  Let $v_{2} = \sigma(v_{1})$.
  By Lemmas~\ref{lem:homogeneous2}(ii), $v_{1}$ is adjacent to all the vertices in $S_{i}$.  Accordingly, $v_{2}$ has no neighbor in $S_{i}$.
  Since $S_{(\ell + 1)/2}$ is a $C_{5}$ , we can find $v_{3}\in V_{(\ell + 1)/2}$ that is adjacent to $v_{2}$ but not $v_{1}$.
  We can conduct the 2-switch $(x v_{1}, v_{2} v_{3}) \rightarrow (x v_{2}, v_{1} v_{3})$ and denote by $G'$ as the resulting graph.
  It can be seen that $S_{(\ell + 1)/2}(G')$ is not a $C_{5}$, contradicting Lemma~\ref{lem:focibly-ds-cycle-scs}(ii). 
\end{proof}

Lemmas~\ref{lem:elementary-graph} and~\ref{lem:focibly-ds-G-scs} imply Theorem~\ref{thm:main} and Rao's characterization of forcibly self-complementary degree sequences~\cite{rao1981survey}.

\begin{theorem}[\cite{rao1981survey}]
  \label{thm:forbibly-core-characterization}
  A degree sequence $(d_i^{n_i})_{i=1}^\ell$ is forcibly self-complementary if and only if for all $i =  1, \dots, \lfloor \ell/2 \rfloor$,
  % \begin{enumerate}[i)]
  % \item
  % \item
  % \end{enumerate}
  %   and if $\ell$ is odd, 
  % \begin{enumerate}[i)]
  % \setcounter{enumi}{2}
  % \item $n_{(\ell + 1)/2} \in \{1,5\}$,
  % \item $d_{(\ell + 1)/2} = {1\over 2} \left(n -1\right)$.
  % \end{enumerate}
  \begin{alignat}{5}
    n_{\ell + 1 -i} &= n_{i} &\in& \{2,4\},
    \\
    d_{\ell + 1 - i} &= n - 1 - d_i &=& \sum_{j=1}^{i} n_j - \frac{1}{2}n_i,
  \end{alignat}
  and $n_{(\ell + 1)/2} \in \{1,5\}$ and $d_{(\ell + 1)/2} = \frac{1}{2} \left(n -1\right)$ when $\ell$ is odd.
\end{theorem}
\begin{proof}
  The sufficiency follows from Lemma~\ref{lem:elementary-graph}: note that an elementary self-complementary pseudo-split graph in which $S_{i}$ has $2 n_{i}$ vertices satisfies the conditions.
 The necessity follows from Lemma~\ref{lem:focibly-ds-G-scs}.
\end{proof}

\section{Enumeration}\label{sec:enumeration}

In this section, we consider the enumeration of self-complementary (pseudo-)split graphs.  The following corollary of Propositions~\ref{prop:scs-odd} and~\ref{lem:split-pseudo-split} focuses us on self-complementary split graphs of even orders.
Let $\lambda_n$ and $\lambda'_n$ denote the number of split graphs and pseudo-split graphs, respectively, of order $n$ that are self-complementary.  For convenience, we set $\lambda_{0} = 1$.
\begin{corollary}
  \label{cor:split-pseudo-split}
  For each $k \ge 1$, it holds $\lambda_{4 k + 1} = \lambda_{4 k}.$
  For each $n > 0$,
  \[
    \lambda'_{n} =
    \begin{cases}
      \lambda_{n} & n \equiv 0 \pmod 4,
      \\
      \lambda_{n - 1} + \lambda_{n - 5} & n \equiv 1 \pmod 4. 
    \end{cases}
  \]
\end{corollary}
\begin{proof}
Proposition~\ref{prop:scs-odd} implies that there exists a one-to-one correspondence between self-complementary split graphs with $4k$ vertices and those with $4k+1$ vertices. 
If a self-complementary pseudo-split graph is not a split graph, then it contains a five cycle and the removal of this five cycle from the graph resulting a self-complementary split graph of an even order by Proposition~\ref{lem:split-pseudo-split}. 
\end{proof}

Let $\sigma = \sigma_1\dots\sigma_p$ be an antimorphism of a self-complementary graph of $4k$ vertices.
We find the number of ways in which edges can be introduced so that the result is a self-complementary split graph with $\sigma$ as an antimorphism.  We need to consider adjacencies among vertices in the same cycle and the adjacencies between vertices from different cycles of $\sigma$. For the second part, we further separate into two cases depending on whether the cycles have the same length. We use $G$ to denote a resulting graph and denote by $G_{i}$ the graph induced by the vertices in the $i$th cycle, for $i = 1, \ldots, p$. By Lemma~\ref{lem:unique-split-partition}, $G$ has a unique split partition and we refer to it as $K \uplus I$.

(i) The subgraph $G_i$ is determined if it has been decided whether $v_{i1}$ is to be adjacent or not adjacent to each of the following $|\sigma_i|\over 2$ vertices in $\sigma_i$. Among those $|\sigma_i|\over 2$ vertices, half of them are odd-numbered in $\sigma_i$. Therefore, $v_{i1}$ is either adjacent to all of them or adjacent to none of them by Lemma~\ref{lem:unique-split-partition}. The number of adjacencies to be decided is ${|\sigma_i|\over 4} + 1$.

(ii) The adjacencies between two subgraphs $G_{i}$ and $G_{j}$ of the same order are determined if it has been decided whether $v_{i1}$ is to be adjacent or not adjacent to each of the vertices in $G_{j}$. By Lemma~\ref{lem:unique-split-partition}, the vertex $v_{i1}$ and half of vertices of $G_{j}$ are decided in $K$ or in $I$ after (i). The number of adjacencies to be decided is $|\sigma_j|\over 2$.

(iii) We now consider the adjacencies between two subgraphs $G_{i}$ and $G_{j}$ of different orders.  We use $\mathrm{gcd}(x, y)$ to denote the greatest common factor of two integers $x$ and $y$.  The adjacencies between $G_{i}$ and $G_{j}$ are determined if it has been decided whether $v_{i1}$ is to be adjacent or not adjacent to each of the first $\mathrm{gcd}(|\sigma_i|, |\sigma_j|)$ vertices of $G_{j}$. Among those $\mathrm{gcd}(|\sigma_i|, |\sigma_j|)$ vertices of $G_{j}$, half of them are decided in the same part of $K \uplus I$ as $v_{i1}$ after (i). The number of adjacencies to be decided is $\frac{1}{2} \mathrm{gcd}(|\sigma_i|, |\sigma_j|)$.

By Lemma~\ref{lem:length-of-circular-permutation}, $|\sigma_i| \equiv 0 \pmod 4$ for every $i = 1, \dots, p$.
Let $c$ be the cycle structure of $\sigma$. We use $c_q$ to denote the number of cycles in $c$ with length $4q$ for every $q = 1, 2, \dots, k$. The total number of adjacencies to be determined is
\begin{align*}
  P &= \sum_{q=1}^{k}(c_{q} (q + 1) +  \frac{1}{2} c_{q} (c_{q}-1) \cdot 2q) + \sum_{1 \leq r < s \leq k} c_{r} c_{s} \cdot \frac{1}{2} \mathrm{gcd}(4r, 4s) \\
  &= \sum_{q=1}^{k} (qc_{q}^{2} + c_{q}) + 2 \sum_{1 \leq r < s \leq k} c_{r}c_{s} \mathrm{gcd}(r, s) \, .
\end{align*}
For each adjacency, there are two choices. Therefore, the number of labeled self-complementary split graphs with this $\sigma$ as an antimorphism is $2^{P}$. 

The number of distinct permutations of the cycle structure $c$ consisting of $c_q$ cycles of length $4q$ for every $q = 1, 2, \dots, k$ is
\begin{equation*}
  \frac{(4k)!}{\prod_{q=1}^{k}(4q)^{c_{q}} \cdot c_{q}!} \, ,
\end{equation*}
and it is the number of possible choices for $\sigma$~\cite{clapham-84-enumeration}.
Let $C_{4k}$ be the set that contains all cycle structures $c$ that satisfy $\sum_{q=1}^{k} c_q \cdot 4q = 4k$. Then the number of antimorphisms with all possible labeled self-complementary split graphs with $4k$ vertices corresponding to each is 
\begin{equation} \label{eq:number-of-labelled-scs-1}
  \sum_{c \in C_{4k}} \frac{(4k!)}{\prod_{q=1}^{k}(4q)^{c_{q}} \cdot c_{q}!} 2^{P} \, .
\end{equation}

For a graph $G$ with $4k$ vertices, let $A_G$ be the set of automorphisms of $G$. Then, the number of different labelings of $G$ is $(4k)! / |A_G|$. If $G$ is self-complementary, then the number of antimorphisms of $G$ is equal to the number of automorphisms of $G$. Let $S$ be the set of all non-isomorphic self-complementary split graphs with $4k$ vertices and let $\lambda_{4k} = |S|$. The number of labeled self-complementary split graphs with all possible antimorphisms corresponding to each is equal to
\begin{equation} \label{eq:number-of-labelled-scs-2}
  \sum_{G \in S} |A_{G}| \frac{(4k)!}{|A_{G}|} = \lambda_{4k} \,(4k)!.
\end{equation}

Let Equation~(\ref{eq:number-of-labelled-scs-1}) equals to Equation~(\ref{eq:number-of-labelled-scs-2}) and we solve for $\lambda_{4k}$:
\begin{equation*}
  \lambda_{4k} = \sum_{c \in C_{4k}} \frac{2^{P}}{\prod_{q=1}^{k}(4q)^{c_{q}} \cdot c_{q}!} \, .
\end{equation*}

In Table~\ref{tbl:pseudo-split}, we list the number of self-complementary (pseudo-)split graphs on up to $21$ vertices.

\begin{table}[h!] 
  \caption{The number of self-complementary (pseudo)-split graphs on $n$ vertices.}
  \label{tbl:pseudo-split}
  \begin{center}
    \small
  \begin{tabular}{ l r r r r r r r r r r } 
  \toprule
    $n$ & 4 & 5 & 8 & 9 & 12 & 13 & 16 & 17 & 20 & 21 \\
  \midrule
    split graphs & 1 & 1 & 3 & 3 & 16 & 16 & 218 & 218 & 9608 & 9608\\
    pseudo-split graphs & 1 & 2 & 3 & 4 & 16 & 19 & 218 & 234 & 9608 & 9826 \\
    all & 1 & 2 & 10 & 36 & 720 & 5600 & 703760 & 11220000 & 9168331776 & 293293716992 \\ 
  \bottomrule
  \end{tabular}
\end{center}
\end{table}

\bibliographystyle{plainurl}
\bibliography{reference}

\begin{thebibliography}{10}

\bibitem{ali-08-thesis}
Parvez Ali.
\newblock {\em Study of Chordal graphs}.
\newblock PhD thesis, Aligarh Muslim University, Aligarh, India, 2008.

\bibitem{cao-23-self-complementary-partition}
Yixin Cao, Haowei Chen, and Shenghua Wang.
\newblock On {Trotignon's} conjecture on self-complementary graphs.
\newblock Manuscript, 2023.

\bibitem{clapham-76-potentially-self-complementary-sequences}
C.~R.~J. Clapham.
\newblock Potentially self-complementary degree sequences.
\newblock {\em J. Combinatorial Theory Ser. B}, 20(1):75--79, 1976.
\newblock \href {https://doi.org/10.1016/0095-8956(76)90069-1}
  {\path{doi:10.1016/0095-8956(76)90069-1}}.

\bibitem{clapham-84-enumeration}
C.~R.~J. Clapham.
\newblock An easier enumeration of self-complementary graphs.
\newblock {\em Proc. Edinburgh Math. Soc. (2)}, 27(2):181--183, 1984.
\newblock \href {https://doi.org/10.1017/S0013091500022276}
  {\path{doi:10.1017/S0013091500022276}}.

\bibitem{clapham-76-self-complementary-degree-sequences}
C.~R.~J. Clapham and D.~J. Kleitman.
\newblock The degree sequences of self-complementary graphs.
\newblock {\em J. Combinatorial Theory Ser. B}, 20(1):67--74, 1976.
\newblock \href {https://doi.org/10.1016/0095-8956(76)90068-x}
  {\path{doi:10.1016/0095-8956(76)90068-x}}.

\bibitem{colbourn-78-isomorphism-and-self-complementary}
Marlene~Jones Colbourn and Charles~J. Colbourn.
\newblock Graph isomorphism and self-complementary graphs.
\newblock {\em SIGACT News}, 10(1):25--29, 1978.
\newblock \href {https://doi.org/10.1145/1008605.1008608}
  {\path{doi:10.1145/1008605.1008608}}.

\bibitem{foldes-77-split-graphs}
St\'ephane Foldes and Peter~L. Hammer.
\newblock Split graphs.
\newblock In {\em Proceedings of the {E}ighth {S}outheastern {C}onference on
  {C}ombinatorics, {G}raph {T}heory and {C}omputing}, pages 311--315, 1977.

\bibitem{gibbs-74-self-complementary}
Richard~A. Gibbs.
\newblock Self-complementary graphs.
\newblock {\em J. Combinatorial Theory Ser. B}, 16:106--123, 1974.
\newblock \href {https://doi.org/10.1016/0095-8956(74)90053-7}
  {\path{doi:10.1016/0095-8956(74)90053-7}}.

\bibitem{harary:1960}
Frank Harary.
\newblock Unsolved problems in the enumeration of graphs.
\newblock {\em Magyar Tud. Akad. Mat. Kutat\'{o} Int. K\"{o}zl.}, 5:63--95,
  1960.

\bibitem{lueker-79-interval-isomorphism}
George~S. Lueker and Kellogg~S. Booth.
\newblock A linear time algorithm for deciding interval graph isomorphism.
\newblock {\em Journal of the ACM}, 26(2):183--195, 1979.
\newblock \href {https://doi.org/10.1145/322123.322125}
  {\path{doi:10.1145/322123.322125}}.

\bibitem{maffray-94-pseudo-split}
Fr\'{e}d\'{e}ric Maffray and Myriam Preissmann.
\newblock Linear recognition of pseudo-split graphs.
\newblock {\em Discrete Appl. Math.}, 52(3):307--312, 1994.
\newblock \href {https://doi.org/10.1016/0166-218X(94)00022-0}
  {\path{doi:10.1016/0166-218X(94)00022-0}}.

\bibitem{rao1981survey}
S.B. Rao.
\newblock A survey of the theory of potentially p-graphic and forcibly
  p-graphic degree sequences.
\newblock In {\em Combinatorics and graph theory}, pages 417--440. Springer,
  1981.

\bibitem{read-63-self-complementary-graphs}
R.~C. Read.
\newblock On the number of self-complementary graphs and digraphs.
\newblock {\em J. London Math. Soc.}, 38:99--104, 1963.
\newblock \href {https://doi.org/10.1112/jlms/s1-38.1.99}
  {\path{doi:10.1112/jlms/s1-38.1.99}}.

\bibitem{ringel-63-self-complementary}
Gerhard Ringel.
\newblock Selbstkomplement\"{a}re {G}raphen.
\newblock {\em Arch. Math. (Basel)}, 14:354--358, 1963.
\newblock \href {https://doi.org/10.1007/BF01234967}
  {\path{doi:10.1007/BF01234967}}.

\bibitem{ryser-57-binary-matrices}
Herbert~John Ryser.
\newblock Combinatorial properties of matrices of zeros and ones.
\newblock {\em Canadian Journal of Mathematics}, 9:371--377, 1957.
\newblock \href {https://doi.org/10.4153/CJM-1957-044-3}
  {\path{doi:10.4153/CJM-1957-044-3}}.

\bibitem{sachs-62-self-complementary-graphs}
Horst Sachs.
\newblock \"{U}ber selbstkomplement\"{a}re {G}raphen.
\newblock {\em Publ. Math. Debrecen}, 9:270--288, 1962.
\newblock \href {https://doi.org/10.5486/pmd.1962.9.3-4.11}
  {\path{doi:10.5486/pmd.1962.9.3-4.11}}.

\bibitem{sridharan-98-self-complementary-chordal}
M.~R. Sridharan and K.~Balaji.
\newblock Characterisation of self-complementary chordal graphs.
\newblock {\em Discrete Math.}, 188(1-3):279--283, 1998.
\newblock \href {https://doi.org/10.1016/S0012-365X(98)00025-9}
  {\path{doi:10.1016/S0012-365X(98)00025-9}}.

\bibitem{trotignon-05-self-complementary-graphs}
Nicolas Trotignon.
\newblock On the structure of self-complementary graphs.
\newblock {\em Electronic Notes in Discrete Mathematics}, 22:79--82, 2005.
\newblock \href {https://doi.org/10.1016/j.endm.2005.06.014}
  {\path{doi:10.1016/j.endm.2005.06.014}}.

\bibitem{xu-00-self-complementary-graphs}
Jin Xu and C.~K. Wong.
\newblock Self-complementary graphs and {R}amsey numbers. {I}. {T}he
  decomposition and construction of self-complementary graphs.
\newblock {\em Discrete Math.}, 223(1-3):309--326, 2000.
\newblock \href {https://doi.org/10.1016/S0012-365X(00)00020-0}
  {\path{doi:10.1016/S0012-365X(00)00020-0}}.

\end{thebibliography}
\end{document}